\documentclass[12pt]{amsart}

\usepackage{amssymb,latexsym}
\usepackage{verbatim,euscript,array}
\usepackage{fullpage,color}
\usepackage{tikz}
\usetikzlibrary{arrows,shapes,trees}

\usepackage[colorlinks=true,linkcolor=blue,urlcolor=my_color,citecolor=redi]{hyperref}

\definecolor{my_color}{rgb}{0,0.5,0.5}
\definecolor{MIXT}{rgb}{0.4,0.3,0.6}
\definecolor{mixt}{rgb}{0.5,0.3,0.2}
\definecolor{sin}{rgb}{0,0.5,0.5}
\definecolor{darkblue}{rgb}{0,0.1,0.8}
\definecolor{redi}{rgb}{0.6,0,0.4}

\input {cyracc.def}
\numberwithin{equation}{section}
\font\tencyr=wncyr10 
\font\tencyi=wncyi10 
\font\tencysc=wncysc10 
\def\rus{\tencyr\cyracc}
\def\rusi{\tencyi\cyracc}
\def\rusc{\tencysc\cyracc}

\newtheorem{thm}{Theorem}[section]
\newtheorem{lm}[thm]{Lemma}
\newtheorem{cl}[thm]{Corollary}
\newtheorem{prop}[thm]{Proposition}

\theoremstyle{remark}
\newtheorem{rmk}[thm]{Remark}

\theoremstyle{definition}
\newtheorem{ex}[thm]{Example}
\newtheorem{df}{Definition}

\newcommand{\eus}{\EuScript}
\newcommand {\ah}{{\mathfrak a}}
\newcommand {\be}{{\mathfrak b}}
\newcommand {\ce}{{\mathfrak c}}

\newcommand {\g}{{\mathfrak g}}
\newcommand {\h}{{\mathfrak h}}
\newcommand {\ka}{{\mathfrak k}}
\newcommand {\el}{{\mathfrak l}}
\newcommand {\me}{{\mathfrak m}}

\newcommand {\p}{{\mathfrak p}}
\newcommand {\q}{{\mathfrak q}}

\newcommand {\es}{{\mathfrak s}}
\newcommand {\te}{{\mathfrak t}}
\newcommand {\ut}{{\mathfrak u}}
\newcommand {\z}{{\mathfrak z}}

\newcommand {\spn}{{\mathfrak{sp}}_{2n}}
\newcommand {\sono}{{\mathfrak{so}}_{2n+1}}
\newcommand {\sone}{{\mathfrak{so}}_{2n}}
\newcommand {\sov}{{\mathfrak{so}}(\eus V)}

\newcommand {\esi}{\varepsilon}
\newcommand {\ap}{\alpha}

\newcommand {\lb}{\lambda}
\newcommand {\vp}{\varphi}

\newcommand {\BV}{{\mathbb V}}
\newcommand {\BQ}{{\mathbb Q}}

\newcommand {\BN}{{\mathbb N}}
\newcommand {\BZ}{{\mathbb Z}}

\newcommand {\ad}{{\mathrm{ad\,}}}

\newcommand {\cha}{{\mathrm{char\,}}}

\newcommand {\hot}{{\mathsf{ht}}}
\newcommand {\ind}{{\mathsf{ind}}}

\newcommand {\Ker}{{\mathrm{Ker\,}}}

\newcommand {\rk}{{\mathsf{rk\,}}}

\newcommand {\spin}{{\mathsf{Spin}}}

\newcommand {\tr}{{\mathsf{tr}}}

\newcommand {\tri}{\mathfrak{sl}_2}
\newcommand {\GR}[2]{{\textrm{{\sf\bfseries #1}}}_{#2}}

\newcommand {\ff}{\boldsymbol{f}}

\newcommand {\ov}{\overline}
\newcommand {\un}{\underline}
\newcommand {\beq}{\begin{equation}}
\newcommand {\eeq}{\end{equation}}

\renewcommand{\le}{\leqslant}
\renewcommand{\ge}{\geqslant}

\newcommand{\bbk}{\Bbbk}

\newenvironment{E6}[6]{
{\footnotesize\begin{tabular}{@{}c@{}}
{#1}---{#2}---\lower3.6ex\vbox{\hbox{\hspace{.6ex}{#3}\rule{0ex}{3.7ex}}
\hbox{\hspace{1ex}\rule{.15ex}{1.5ex}\rule{0ex}{1.4ex}}\hbox{{#6}\strut}}---{#4}---{#5}
\end{tabular}}}

\newenvironment{E7}[7]{
{\footnotesize\begin{tabular}{@{}c@{}}
{#1}---{#2}---{#3}---\lower3.6ex\vbox{\hbox{\hspace{.6ex}{#4}\rule{0ex}{3.7ex}}
\hbox{\hspace{1ex}\rule{.15ex}{1.5ex}\rule{0ex}{1.4ex}}\hbox{{#7}\strut}}---{#5}---{#6}
\end{tabular}}}

\newenvironment{E8}[8]{%
{\footnotesize\begin{tabular}{@{}c@{}}
{#1}---{#2}---{#3}---{#4}---\lower3.6ex\vbox{\hbox{\hspace{.6ex}{#5}\rule{0ex}{3.7ex}}
\hbox{\hspace{1ex}\rule{.15ex}{1.5ex}\rule{0ex}{1.4ex}}\hbox{{#8}\strut}}---{#6}---{#7}
\end{tabular}}}

\begin{document}
\setlength{\parskip}{2pt plus 4pt minus 0pt}
\hfill {\scriptsize December 1, 2019} 
\vskip1.5ex

\title[Casimir elements and Levi subalgebras]%
{Casimir elements associated with Levi subalgebras of simple Lie algebras and their applications} 
\author{Dmitri I. Panyushev}
\address[]{Institute for Information Transmission Problems of the R.A.S., Bolshoi Karetnyi per. 19,  
127051 Moscow,  Russia}
\email{panyushev@iitp.ru}
\keywords{$\BZ$-grading, root system, isotropy representation, abelian subspace}
\subjclass[2010]{17B20, 17B22, 15A75}
\thanks{This research is partially supported by the R.F.B.R. grant {\rus N0} 16-01-00818}

\begin{abstract}
Let $\g$ be a simple Lie algebra, $\h$ a Levi subalgebra, and $\eus C_\h\in \eus U(\h)$ the Casimir 
element defined via the restriction of the Killing form on $\g$ to $\h$. We study 
$\eus C_{\h}$-eigenvalues in $\g/\h$ and related $\h$-modules. Without loss of generality, one may 
assume that $\h$ is a maximal Levi. Then $\g$ is equipped with the natural $\BZ$-grading 
$\g=\bigoplus_{i\in\BZ}\g(i)$ such that $\g(0)=\h$ and  $\g(i)$ is a simple 
$\h$-module for $i\ne 0$.
We give explicit formulae for the $\eus C_\h$-eigenvalues in each $\g(i)$, $i\ne 0$ and relate 
eigenvalues of $\eus C_\h$ in $\bigwedge^\bullet\g(1)$ to the dimensions of abelian subspaces of 
$\g(1)$. Then we prove that if  $\ah\subset\g(1)$ is abelian, 
whereas $\g(1)$ is not, then $\dim\ah\le \dim\g(1)/2$. Moreover, if $\dim\ap=(\dim\g(1))/2$,
then $\ah$ has an abelian complement.
The $\BZ$-gradings of height $\le 2$ are closely related to involutions of $\g$, and we provide a 
connection of our theory to (an extension of) the ``strange formula'' of Freudenthal--de Vries.
\end{abstract}
\maketitle

\tableofcontents
\section*{Introduction}

\noindent
Let $G$ be a simple algebraic group with Lie algebra $\g$, $\eus U(\g)$ the enveloping algebra, and 
$\Phi$ the Killing form on $\g$. If $\h\subset\g$ is a reductive subalgebra, then $\Phi\vert_\h$ is 
non-degenerate and $\me:=\h^\perp$ is a complementary $\h$-submodule of $\g$, i.e., 
$\g=\h\oplus\me$. Using $\Phi\vert_\h$, one defines 
the Casimir element $\eus C_\h\in \eus U(\h)$, and our goal is to study $\eus C_{\h}$-eigenvalues  in 
$\me$ and related $\h$-modules. In~\cite{jlms01}, we proved that {\sf (i)} the $\eus C_\h$-eigenvalues 
in $\me$ do not exceed $1/2$ and {\sf (ii)} if $\h$ is the fixed-point subalgebra of an involution, i.e., 
$[\me,\me]\subset\h$, then $\eus C_\h$ acts scalarly on $\me$, as $\frac{1}{2}{\sf id_\me}$. First, we
prove a complement to it. Namely, if $\eus C_\h$ does have an eigenvalue $1/2$ in $\me$, then 
$[\me,\me]\subset\h$ and thereby `$1/2$' is the only $\eus C_\h$-eigenvalue on $\me$.
\\ \indent
Then we stick to the case in which $\h$ is a Levi subalgebra of $\g$. Let 
$\te\subset\h$ be a Cartan subalgebra and $\Delta$ (resp. $\Delta_\h$) the root system of 
$(\g,\te)$ (resp. $(\h,\te)$). Let $\be_\h$ be a Borel subalgebra of $\h$ containing $\te$ and $\be$ 
a Borel subalgebra of $\g$ such that $\be\cap\h=\be_\h$. This yields the sets of positive roots 
$\Delta^+_\h\subset\Delta^+\subset\Delta$ and decomposition $\g=\me^-\oplus\h\oplus\me^+$, where 
$\be=\be_\h\oplus\me^+$. Then $\p:=\h\oplus\me^+$ is a standard parabolic subalgebra and
$\Delta^+=\Delta^+_\h\cup \Delta(\me^+)$, where $\Delta(\me^+)$ is the set of $\te$-weights of 
$\me^+=\p^{\sf nil}$.
Let $\Pi$ be the set of simple roots in $\Delta^+$ and $\Pi_\h:=\Pi\cap\Delta^+_\h$.
If $k=\#(\Pi\cap \Delta(\me^+))$, then $\g$ is equipped with a natural $\BZ^k$-grading.
While studying $\eus C_\h$-eigenvalues in $\me$, one may assume that $\h$ is a maximal Levi, 
i.e., $k=1$, see Section~\ref{subs:Z-grad-versus} for details.
For $\Pi\cap \Delta(\me^+)=\{\ap\}$, the corresponding $\BZ$-grading 
is called the $(\BZ,\ap)$-{\it grading}. Let $\g=\bigoplus_{i\in\BZ}\g_\ap(i)$ denote this grading, where
$\h=\g_\ap(0)$ and $\me^+=\bigoplus_{i\ge 1}\g_\ap(i)=:\g_\ap({\ge}1)$. In this case, $\ap$ is the lowest
weight of the simple $\g_\ap(0)$-module $\g_\ap(1)$. Moreover, {\bf each} $\g(i)$, $i\ne 0$, is a simple 
$\g(0)$-module~\cite[Chap.\,3,\,\S3.5]{t41},\,\cite[Theorem\,0.1]{ko10}. Then we write
$\eus C_\ap(0)$, $\be_\ap(0)$, $\p_\ap$ in place of $\eus C_{\g_\ap(0)}$, $\be_{\g_\ap(0)}$, $\p$,
respectively. 
\\ \indent
Using the partition of $\Delta^+(\me)$ associated with the $(\BZ,\ap)$-grading,
we obtain explicit formulae for the $\eus C_\ap(0)$-eigenvalue in any $\g_\ap(i)$, $i\ne 0$.
Let $\vp_\ap$ be the fundamental weight of $\g$ corresponding to $\ap$ and $\Delta_\ap(i)$ the set of roots of $\g_\ap(i)$. The sum of all elements of 
$\Delta_\ap(i)$, denoted $|\Delta_\ap(i)|$, is a multiple of $\vp_\ap$, i.e., $|\Delta_\ap(i)|=q_\ap(i)\vp_\ap$ and $q_\ap(i)\in\BN$.
Hence $|\Delta_\ap({\ge}1)|=(\sum_{i\ge 1}q_\ap(i))\vp_\ap=:q_\ap\vp_\ap$ is the sum of all roots in the 
nilradical of $\p_\ap$. Set $r_\ap:=\|\theta\|^2/\|\ap\|^2\in \{1,2,3\}$, where $\theta\subset\Delta^+$ is the 
highest root, and let $h^*$ be the {\it dual Coxeter number} of $\g$. 
Let $\gamma_\ap(k)$ denote the $\eus C_\ap(0)$-eigenvalue on $\g_\ap(k)$.

\begin{thm}     \label{thm:gamma(k)}
We have $\gamma_\ap(k)=\displaystyle\frac{k}{2h^*r_\ap}\sum_{i\ge 1}q_\ap(ki)$.
\emph{(In particular, $\gamma_\ap(1)=
q_\ap/2h^*r_\ap$.)}
\end{thm}
\noindent
We also obtain a series of relations between numbers $\gamma_\ap(i), q_\ap(i), \dim\g_\ap(i)$.
For instance, if $d_\ap=\max\{i\mid \Delta_\ap(i)\ne \varnothing\}$, then
$\gamma_\ap(d_\ap)=1-d_\ap\gamma_\ap(1)$ and $q_\ap+q_\ap(d_\ap)=2h^*r_\ap/d_\ap$.

Let $\delta_\ap(k)$ be the {\bf maximal} $\eus C_\ap(0)$-eigenvalue in $\bigwedge^k\g_\ap(1)$, so that 
$\delta_\ap(1)=\gamma_\ap(1)$. We relate the values $\{\delta_\ap(i)\mid i=1,2,\dots\}$ to dimensions of 
abelian subspaces of $\g_\ap(1)$ as follows.

\begin{thm}    \label{thm:delta(k)}
For each $k=1,2.\dots,\dim\g_\ap(1)$, we have $\delta_\ap(k)\le k\gamma_\ap(1)$. This upper bound 
is attained for a given $k$ if and only if\/ $\g_\ap(1)$ contains a $k$-dimensional abelian subspace.
\end{thm}
\noindent
Similar results are obtained earlier for abelian subspaces of $\g$~\cite{ko65} and for abelian subspaces
related to certain $\BZ_m$-gradings of $\g$~\cite{jlms01}.
One of the applications is that if $\g_\ap(1)$ is not abelian (which exactly means that $d_\ap>1$) 
and $\ah\subset\g_\ap(1)$ is an abelian subspace, then $\dim\ah\le (1/2)\dim\g_\ap(1)$. A related
result is that if there is an abelian subspace $\ah\subset\g_\ap(1)$ of dimension $(1/2)\dim\g_\ap(1)$, 
then {\sf\bfseries (1)} $\ah$ has an abelian complement; {\sf\bfseries (2)} all the numbers 
$\{\delta_\ap(i)\}$ can explicitly be computed. It appears here that the sequence 
$\delta_\ap(1),\dots,\delta_\ap(m)$ has an interesting behaviour that is governed by a relation
between $q_\ap$ and $q_\ap(1)$. We also provide some methods for constructing abelian subspaces 
of $\g_\ap(1)$ and point out the maximal dimension of an abelian subspace in $\g_\ap(1)$ for {\bf all} 
$(\BZ,\ap)$-gradings. The latter is related to a recent work of Elashvili et al.~\cite{e-j-k}.

For an involution $\sigma$ of $\g$, let $\g=\g_0\oplus\g_1$ be the associated $\BZ_2$-grading and
$\eus C_0\in \eus U(\g_0)$ the Casimir element defined via $\Phi\vert_{\g_0}$. Then the 
$\eus C_0$-eigenvalue on $\g_1$ equals $1/2$~\cite{jlms01}. As $\g_1$ is an orthogonal $\g_0$-module, 
there is a natural $\g_0$-module $\spin(\g_1)$ related to the exterior algebra of $\g_1$~\cite{tg01},
see Section~\ref{sect:FdV} for details.
Although $\spin(\g_1)$ is often reducible, $\eus C_0$ acts scalarly on it, and the corresponding 
eigenvalue, $\gamma_{\spin(\g_1)}$,  is computed in~\cite[Theorem~7.7]{tg01}, cf. Section~\ref{sect:FdV}. Here we obtain 
another uniform expression.

\begin{thm}    \label{thm:1/16}
For any involution (=$\BZ_2$-grading) of $\g$, one has  $\gamma_{\spin(\g_1)}=(\dim\g_1)/16$.
\end{thm}

\noindent 
The {\bf inner} involutions are closely related to $(\BZ,\ap)$-gradings with $d_\ap\le 2$~\cite{kac}, and 
in this case we give a proof of Theorem~\ref{thm:1/16} that uses properties of 
$\eus C_\ap(0)$-eigenvalues. However, the argument that exploits $(\BZ,\ap)$-gradings does not 
extend to outer involutions. Our general proof invoke the "strange formula" of Freudenthal--de Vries, 
which asserts that $(\rho,\rho)=(\dim\g)/24$~\cite[47.11]{FdV}, where $2\rho=|\Delta^+|$. 
On the other hand, the adjoint representation of $\g$ occurs as the isotropy representation related to 
the involution $\tau$ of $\g\dotplus\g$ with $\tau(x,y)=(y,x)$. Although $\g\dotplus\g$ is not simple, one 
can state an analogue of Theorem~\ref{thm:1/16} for $(\g\dotplus\g,\tau)$, and we prove that that 
analogue is equivalent to the "strange formula". It is important here that, for the orthogonal $\g$-module 
$\g$, one has $\spin(\g)=2^{\rk\g/2]}\eus V_\rho$, where $\eus V_\rho$ is the simple $\g$-module with highest weight $\rho$. 
This result of Kostant appears in~\cite[p.\,358]{ko61}, cf. also \cite[Sect.\,5]{ko97}. To a great extent,
our general study of `$\spin(V)$' in \cite{tg01} was motivated by that observation.

The paper is structured as follows. In Section~\ref{sect:prelim-Cas}, we recall basic facts on Casimir 
elements, the Dynkin index of a simple subalgebra of $\g$, and $\BZ$-gradings. In 
Section~\ref{sect:some-prop}, we discuss some properties of $(\BZ,\ap)$-gradings and numbers
$\{q_\ap\}_{\ap\in\Pi}$. Section~\ref{sect:5/2} contains our results
on Theorem~\ref{thm:gamma(k)} and the $\eus C_\ap(0)$-eigenvalues in 
$\g_\ap(i)$.  
In Sections~\ref{sect:3} and \ref{sect:applic}, we study maximal eigenvalues of 
$\eus C_\ap(0)$ in $\g_\ap(0)$-modules $\bigwedge^i \g_\ap(1)$ ($1\le i\le\dim\g_\ap(1)$) and their relationship to abelian subspaces of 
$\g_\ap(1)$. Section~\ref{sect:FdV} is devoted to connections between $\BZ_2$-gradings and
$(\BZ,\ap)$-gradings with $d_\ap\le 2$. 
Here we discuss the "strange formula" and a generalisation of it to the $\BZ_2$-graded situation. In Appendix~\ref{sect:App}, we gather the 
tables of eigenvalues $\gamma_\ap(i)$ and numbers $q_\ap(i)$ for all $(\BZ,\ap)$-gradings.
\\ \indent
The ground field $\bbk$ is algebraically closed and $\cha\bbk=0$. We use `$\dotplus$' to denote the direct sum of Lie algebras.

\section{Casimir elements, Levi subalgebras and gradings}
\label{sect:prelim-Cas}

\noindent
Unless otherwise stated, $\g$ is a simple Lie algebra with a fixed triangular decomposition
$\g=\ut\oplus\te\oplus\ut^-$ and $\Phi$ is the Killing form on $\g$. Then $\Delta$ is the root system of 
$(\g,\te)$ and $\Delta^+$ is the set of positive roots corresponding to $\be=\te\oplus\ut$. Let 
$\Pi=\{\ap_1,\dots,\ap_n\}$ be a set of simple roots in $\Delta^+$, $\{\vp_1,\dots,\vp_n\}$  the 
corresponding set of fundamental weights, and $\theta$ the {\it highest root\/} in $\Delta^+$. We also 
write $\vp_\ap$ for the fundamental weight corresponding to $\ap\in\Pi$.

\subsection{The Casimir element associated with a reductive subalgebra}
Let  $\h$ be a reductive algebraic subalgebra of $\g$. Then $\Phi\vert_\h$ is 
non-degenerate~\cite[Chap.\,1,\,\S\,6.3]{t41} and one defines the Casimir element $\eus C_\h$. 
Namely, if $\{e_i\}$ and $\{e'_i\}$ are the dual bases of $\h$ 
w.r.t{.} $\Phi\vert_\h$, then $\eus C_\h:=\sum_{i=1}^{\dim\h} e'_i e_i\in\eus U(\h)$.
As is well known, $\eus C_\h$ is a well-defined quadratic element of the centre of $\eus U(\g)$ and 
the eigenvalues of $\eus C_\h$ on finite-dimensional $\h$-modules are non-negative rational numbers, cf.~\cite[Chap.\,3,\S\,2.9]{t41}.
We have $\g=\h\oplus\me$, where $\me=\h^\perp$ is an $\h$-module. 

\begin{prop}[cf.\,{\cite[Theorem\,2.3]{jlms01}}]    \label{prop:lms01}
{\sf (i)} \  $\tr_\g(\eus C_\h)=\dim\h$; \\
{\sf (ii)} \   If $x,y\in\h$, then $\Phi(\eus C_\h(x),y)=\tr_\h(\ad(x)\ad(y))$; \\
{\sf (iii)} \   Any $\eus C_\h$-eigenvalue in $\me$ is at most $1/2$. Moreover, if this bound is attained and
$\me_{1/2}\ne 0$ is the corresponding eigenspace, then $[\me_{1/2},\me]\subset \h$. \\
{\sf (iv)} \ If\/ $\g=\h\oplus\me$ is a $\BZ_2$-grading (i.e., $[\me,\me]\subset\h$), then $\me=\me_{1/2}$.
\end{prop}

The following is a useful complement to the above properties.

\begin{prop}      \label{prop:m(1/2)}
Given $\g=\h\oplus\me$ and $\eus C_\h$ as above, suppose that  
$\me_{1/2}\ne 0$. Then  $\me_{1/2}=\me$ and thereby the decomposition $\g=\h\oplus\me$ is a $\BZ_2$-grading.
\end{prop}
\begin{proof}
Write $\me=\me_{1/2}\oplus \tilde\me$, where $\tilde\me$ is the sum of all other eigenspaces of 
$\eus C_\h$ in $\me$. One has $\Phi(\eus C_\h(x),y)=\Phi(x,\eus C_\h(y))$ for all $x,y$. Hence
$\Phi(\me_{1/2},\tilde\me)=0$ and $\Phi$ is non-degenerate on $\tilde \h:=\h\oplus \me_{1/2}$.
Therefore, $\tilde\h$ is reductive and $\tilde\me={\tilde\h}^\perp$ is a $\tilde\h$-module. On the other
hand, $[\me_{1/2},\tilde\me]\subset \h$, see Prop.~\ref{prop:lms01}(iii). Hence $[\me_{1/2},\tilde\me]=0$.
Let $\hat\h$ be the subalgebra of $\tilde\h$ generated by $\me_{1/2}$. Then
$[\h,\hat\h]\subset \hat\h$ and also $[\me_{1/2},\hat\h]\subset\hat\h$, i.e., $\hat\h$ is an ideal of $\tilde\h$.
We can write $\tilde\h=\hat\h\oplus\es$, where $\es$ is a complementary ideal. Then
$\g=\es\oplus\hat\h\oplus\tilde\me$, $[\es,\hat\h]=0$, and $[\hat\h,\tilde\me]=0$. Therefore
$\hat\h$ is an ideal of $\g$. Thus $\hat\h=\g$ and $\es=\tilde\me=0$. 
\end{proof}

For $\h=\g$, one obtains the usual Casimir element $\eus C=\eus C_\g\in \eus U(\g)$.
Let $(\ ,\ )$ denote the {\it canonical bilinear form\/} on $\te^*$, i.e., 
one induced by the restriction of $\Phi$ to $\te$, see~\cite[Chap.\,6,\,\S\,1, n$^o$\,12]{bour} for its 
properties.  If $\eus V_\lb$ is a simple $\g$-module with highest weight  $\lb$, then $\eus C$ acts on 
$\eus V_\lb$ scalarly with eigenvalue $(\lb,\lb+2\rho)$~\cite[Chap.\,3, Prop.\,2.4]{t41}.
Since $\eus C(x)=x$ for any $x\in \g$, this means that $(\theta,\theta+2\rho)=1$. The latter is
equivalent to that $(\theta,\theta)=1/h^*$, where $h^*$ is the dual Coxeter number of $\g$,
cf. e.g.~\cite[1.1]{jlms01}.  And the "strange formula" of Freudenthal--de Vries asserts that 
$(\rho,\rho)=(\dim\g)/24$, see~\cite[47.11]{FdV}.

\subsection{The transition factor and the Dynkin index}    
\label{subs:transit-factor}
Let $\ka\subset\g$ be a {\bf simple} subalgebra and $\Phi_\ka$ the Killing form on $\ka$. Then
$\Phi\vert_\ka$ is proportional to $\Phi_\ka$, i.e., there is $F\in\BQ$ such that 
$\Phi(x,x)=F\cdot \Phi_\ka(x,x)$ for any $x\in\ka$. The transition factor $F$ can be expressed via the other known objects. Consider an invariant bilinear form $(\ {|}\ )_\g$ on $\g$, normalised as follows. Let 
$\langle\ {,}\ \rangle_\g$ be the induced $W$-invariant bilinear form on $\te^*$. Following Dynkin,
we then require that
$\langle\theta,\theta\rangle_\g=2$; and likewise for $(\ {|}\ )_\ka$ and $\langle\ {,}\ \rangle_\ka$.

\begin{df}[cf.~{\cite[n$^o$\,7]{dy}}]     \label{def:D-ind}
The {\it Dynkin index\/} of a simple subalgebra $\ka$ in $\g$ is defined to be 
$\ind(\ka\hookrightarrow\g):=\displaystyle\frac{(x|x)_\g}{(x|x)_\ka}$ \ for $x\in\ka$.
\end{df}

The following simple assertion is left to the reader. For a non-degenerate symmetric bilinear form 
$\Psi$ on $\BV$, let $\Psi^*$ denote the induced bilinear form on $\BV^*$.
\begin{lm}            
\label{lm:dve-formy}
If\/ $\Psi_1$ and $\Psi_2$ are two such forms and $\Psi_1=f\Psi_2$ for some $f\in \bbk^\times$, 
then  $\Psi_2^*=f\Psi_1^*$.
\end{lm}

\noindent Using this, we give a formula for the transition factor $F$
between $\Phi$ and $\Phi_\ka$ 
or, rather, the transition factor $T$ between the induced canonical bilinear forms $(\ ,\ )$ on $\te^*$ and $(\ ,\ )_\ka$ on $\te^*_\ka$, where $\te_\ka$ is a suitable Cartan subalgebra of $\ka$ and we regard $\te^*_\ka$ as subspace of $\te^*$. 

\begin{prop}     \label{prop:transition}
{\sf (i)} \ The transition factor between $(\ ,\ )$ and $(\ ,\ )_\ka$ \ is 
$T=\displaystyle \frac{1}{F}=\frac{h^*(\ka)}{h^*{\cdot} \ind(\ka\hookrightarrow\g)}$.
\\
{\sf (ii)} \ Furthermore, $\ind(\ka\hookrightarrow\g)=\displaystyle \frac{(\theta,\theta)}{(\ov{\theta},\ov{\theta})}$, where $\ov{\theta}$ is the highest root of\/  $\ka$.
\end{prop}
\begin{proof}
{\sf (i)}  Using Lemma~\ref{lm:dve-formy} and Def.~\ref{def:D-ind}, we notice that $T=
\displaystyle\frac{1}{F}$ and
$\displaystyle
\ind(\ka\hookrightarrow\g)=\frac{\langle\nu,\nu\rangle_\ka}{\langle\nu,\nu\rangle_\g}$ for
any $\nu\in\te^*_\ka$.
Since $(\theta,\theta)=1/h^*$ and $\langle\theta,\theta\rangle_\g=2$, we have
$(\ ,\ )=2h^*\langle\ ,\ \rangle_\g$ and likewise for two forms on $\te^*_\ka$. Then for any 
$\nu\in\te^*_\ka\subset\te^*$, we obtain
\[
    T=\frac{(\nu,\nu)}{(\nu,\nu)_\ka}=\frac{(\nu,\nu)}{\langle\nu,\nu\rangle_\g}{\cdot}
    \frac{\langle\nu,\nu\rangle_\g}{\langle\nu,\nu\rangle_\ka}{\cdot} 
    \frac{\langle\nu,\nu\rangle_\ka}{(\nu,\nu)_\ka}=
 \frac{1}{2h^*}\frac{1}{\ind(\ka\hookrightarrow\g)}{\cdot}2h^*(\ka)
 =\frac{h^*(\ka)}{h^*{\cdot} \ind(\ka\hookrightarrow\g)}.  
\]
{\sf (ii)} \ Taking $\nu=\ov{\theta}$, we obtain
\[
   \ind(\ka\hookrightarrow\g)=\frac{\langle\ov{\theta},\ov{\theta}\rangle_\ka}{\langle\ov{\theta},\ov{\theta}\rangle_\g}=\frac{2}{\langle\ov{\theta},\ov{\theta}\rangle_\g}=
   \frac{\langle{\theta},{\theta}\rangle_\g}{\langle\ov{\theta},\ov{\theta}\rangle_\g}=
   \frac{(\theta,\theta)}{(\ov{\theta},\ov{\theta})}.   \qedhere
\]
\end{proof}

\subsection{Levi subalgebras and gradings}     \label{subs:Z-grad-versus}
By definition, a Levi subalgebra is the centraliser in $\g$ of a toral subalgebra (i.e., of the Lie algebra of 
an algebraic torus). If $\h=\z_\g(\tilde\ce)$ for a toral subalgebra $\tilde\ce$, then $\ce:=\z_\g(\h)$ is 
the centre of $\h$ and $\h=\ce\dotplus \es$, where $\es=[\h,\h]$. For  $\mu\in\ce^*$, set
$\me(\mu)=\{x\in\me\mid [c,x]=\mu(c)x \ \ \forall  c\in \ce\}$. Then $\me(\mu)$ is an $\h$-module. By an
old result of Kostant,  $\me(\mu)$ is a {\bf simple} $\h$-module. See~\cite[p.\,136]{ko10}
for a proof and historical remarks. (An alternate independent approach appears 
in~\cite[Chap.\,3,\,\S 3.5]{t41}.) 
As in the introduction, we assume that $\te\subset\h$ and  $\be_\h\subset\be$. This provides the 
decomposition $\g=\me^-\oplus\h\oplus\me^+$ and partition $\Delta^+_\h\cup\Delta(\me^+)=\Delta^+$. 
If $\dim\ce=k$, then one defines a $\BZ^k$-grading of $\g$ as follows. To simplify notation,  
assume that $\Pi\cap\Delta(\me^+)=\{\ap_1,\dots,\ap_k\}$. For $\gamma\in \Delta$, let $\g^\gamma$ 
denote the corresponding root space. If $\gamma=\sum_{i=1}^na_i\ap_i\in \Delta$, then the 
$\ap_i$-{\it height\/} of $\gamma$ is $\hot_{\ap_i}(\gamma)=a_i$ and $\hot(\gamma)=\sum_{i}a_i$ is
the (usual) {\it height} of $\gamma$. For a 
$k$-tuple $(j_1,\dots,j_k)\in\BZ^k$, set 
\[ 
\Delta(j_1,\dots,j_k)=\{\gamma\in\Delta\mid \hot_{\ap_i}(\gamma)=j_i, \ 1\le i\le k \} \ \text{ and } \
   \g(j_1,\dots,j_k)=\bigoplus_{\gamma\in \Delta(j_1,\dots,j_k)} \g^\gamma .
\]
This yields a $\BZ^k$-grading $\g=\bigoplus_{j_1,\dots,j_k} \g(j_1,\dots,j_k)$ 
with $\g(0,\dots,0)=\h$. By the above result of Kostant, each $\g(j_1,\dots,j_k)$
with $(j_1,\dots,j_k)\ne (0,\dots,0)$ is a simple $\h$-module. Indeed, if $(\nu_i,\ap_j)=\delta_{ij}$, 
$1\le i,j\le k$ and $\mu=\sum_{i=1}^k j_i\nu_i\in\ce^*$, then $\g(j_1,\dots,j_k)=\me(\mu)$.
If $k=1$ and $\Pi\cap\Delta(\me^+)=\{\ap\}$, then $\h$ is a maximal Levi and the corresponding 
$\BZ$-grading is called the $(\BZ,\ap)$-{\it grading}. In this case, we write $\g_\ap(j)$ in place of $\g(j)$.

The passage from an arbitrary Levi subalgebra $\h\subset\g$ to a maximal Levi subalgebra of a simple 
subalgebra of $\g$ goes as follows. Suppose that we are to compute the $\eus C_\h$-eigenvalue on
a simple $\h$-module $V=\g(j_1,\dots,j_k)\subset\me^+$. Here $V^*=\g(-j_1,\dots,-j_k)\subset \me^-$  is 
the dual $\h$-module and $\Phi$ is non-degenerate on $V\oplus V^*$. Take
\[
  \q=\textstyle \bigoplus_{i\in\BZ}\g(ij_1,\dots,ij_k) \subset \g .
\]
It is a $\BZ$-graded subalgebra of $\g$ with $\q(i)=\g(ij_1,\dots,ij_k)$. Since $\Phi\vert_\q$ is 
non-degenerate, $\q$ is reductive. 
Furthermore, by~\cite[Sect.\,1]{ko10}, the positive part 
$\q({\ge}1)$ is generated by $V=\q(1)$. Since each $\q(i)$, $i\ne 0$, is a simple $\q(0)$-module,
the $\BZ$-grading of $\q$ is determined by a sole simple root of $\q$. Taking the corresponding simple ideal of $\q$, one can write
$\q=\ka\dotplus \el$, where $\el$ is reductive, $\ka$ is simple, and there is a simple root $\beta$ of $\ka$ 
such that $\q(i)=\ka_\beta(i)$ for $i\ne 0$, while 
\beq    \label{g(0)-decomp}
    \h=\q(0)=\ka_\beta(0)\dotplus \el .
\eeq
Thus, $\ka_\beta(0)$ is a maximal Levi subalgebra of $\ka$ and $V=\ka_\beta(1)$ for the $(\BZ{,}\beta)$-grading of $\ka$. Taking a basis for $\h$ adapted to 
the sum in~\eqref{g(0)-decomp}, one can split $\eus C_\h$ as
$\eus C_{\h}=\tilde{\eus C}_{\ka_\beta(0)}+\tilde{\eus C}_\el$.
Since $\el$ acts trivially on $V$, the eigenvalues of $\eus C_{\h}$ and $\tilde{\eus C}_{\ka_\beta(0)}$ 
on $V$ are the same. Furthermore, if ${\eus C}_{\ka_\beta(0)}$ is the true Casimir element associated 
with $(\Phi_\ka)\vert_{\ka_\beta(0)}$
and $\Phi\vert_\ka=F{\cdot}\Phi_\ka$ (cf. Section~\ref{subs:transit-factor}), 
then $\tilde{\eus C}_{\ka_\beta(0)}=F{\cdot}{\eus C}_{\ka_\beta(0)}$. Here the factor $F$ comes from the
fact that the dual bases for $\ka_\beta(0)$ required in the Casimir elements 
$\eus C_{\h}$ (i.e., in $\tilde{\eus C}_{\ka_\beta(0)}$) and ${\eus C}_{\ka_\beta(0)}$
are being computed via the proportional bilinear forms $\Phi\vert_\ka$ and $\Phi_\ka$, respectively. Thus,
\\ \indent
{\sl  for any simple $\h$-module $V\subset\me^+$, there is a simple $\BZ$-graded subalgebra 
$\ka\subset \g$ and a simple root $\beta$ of\/ $\ka$ such that $V=\ka_\beta(1)$ and then the
$\eus C_{\h}$-eigenvalue in $V$ equals  $F$ times the ${\eus C}_{\ka_\beta(0)}$-eigenvalue on $V$,
where $\Phi\vert_\ka=F{\cdot}\Phi_\ka$.
}
\\ \indent
For this reason, we restrict ourselves with considering only maximal Levi subalgebras of $\g$ and
the corresponding $(\BZ,\ap)$-gradings.

\section{$(\BZ,\ap)$-gradings and partitions of root systems}
\label{sect:some-prop}

\noindent
If $\Delta$ has two root lengths, then $\Pi_l$ is the set of {\bf long} simple roots and $\theta_s$ stands 
for the dominant {\bf short} root. In the simply-laced case, we assume that $\Pi_l=\Pi$ and 
$\theta_s=\theta$. Recall that $\hot_{\ap_i}(\gamma)$  is the $\ap_i$-{height} of $\gamma\in\Delta$.
Given $\ap\in\Pi$, set $d_\ap=\hot_\ap(\theta)$, 
$\Delta_{\ap}(i)=\{\gamma \mid \hot_{\ap}(\gamma)=i\}$,  
$\Delta^+_{\ap}(0)=\Delta^+\cap\Delta_{\ap}(0)$, and
\[
  \mathcal R_\ap=\textstyle \bigsqcup_{i=1}^{d_\ap} \Delta_{\ap}(i)=\{\gamma\mid (\gamma,\vp_\ap)>0\} .
\]
Then $ \Delta=\Delta^+_{\ap}(0)\sqcup \mathcal R_\ap$, 
and $\Delta_{\ap}(i)$ is the set of $\te$-weights of $\g_\ap(i)$, where
$\g=\bigoplus_{i=-d_\ap}^{d_\ap}\g_\ap(i)$ is the $(\BZ,\ap)$-grading. Since $\g_\ap(i)$ is a simple 
$\g_\ap(0)$-module for $i\ne 0$,   
$\Delta_\ap(i)$ with $i> 0$ contains a unique minimal root (=\,the lowest weight of $\g_\ap(i)$ w.r.t. 
$\Delta_\ap^+(0)$) and a unique maximal root (=\,the highest weight).

As usual, $\gamma^\vee=2\gamma/(\gamma,\gamma)$ and 
$\Delta^\vee=\{\gamma^\vee\mid \gamma\in \Delta\}$ is the {\it dual root system}.  
The set of simple roots in $(\Delta^+)^\vee$ is $\Pi^\vee$ and notation $\hot(\gamma^\vee)$ refers to 
the height of $\gamma^\vee$ in $\Delta^\vee$. The fundamental weight $\vp_\ap$ is {\it minuscule}, if 
$(\vp_\ap,\gamma^\vee)\le 1$ for any $\gamma\in \Delta^+$, 
i.e., $(\vp_\ap,\theta_s^\vee)=1$; and $\vp_\ap$ is {\it cominuscule}, if 
$\hot_{\ap}(\theta)=1$, i.e., $d_\ap=1$.

\un{\it\bfseries Coxeter numbers}.  Set $h=h(\g):=\hot(\theta)+1$---the {\it Coxeter number\/} of $\g$ and 
$h^*=h^*(\g):=\hot(\theta^\vee)+1$---the {\it dual Coxeter number\/} of $\g$. Since $\theta_s^\vee$ is 
the highest root in $\Delta^\vee$,  we have $\hot(\theta_s)+1=h^*(\g^\vee)$---the dual Coxeter number of 
the Langlands dual Lie algebra $\g^\vee$. Note that $h(\g)=h(\g^\vee)$, hence $h^*\le h$.
However, $h^*(\g)$ and  $h^*(\g^\vee)$ can be different.
Thus, there are up to three Coxeter numbers for $(\g,\g^\vee)$, which all coincide in the {\sf\bfseries ADE}-case.

If $M\subset\Delta^+$, then $|M|=\sum_{\gamma\in M}\gamma$, while $\#M$ stands for the cardinality. 
As usual, $2\rho=|\Delta^+|$ and hence $(\rho,\gamma^\vee)=\hot(\gamma^\vee)$ for any 
$\gamma\in \Delta^+$. The orthogonal projection of $2\rho$ to the edge of the Weyl chamber 
corresponding to $\vp_\ap$ can be written as $q_\ap\vp_\ap$ and it is  clear that 
$\displaystyle q_\ap=\frac{(2\rho,\vp_\ap)}{(\vp_\ap,\vp_\ap)}$. The numbers $\{q_\ap\}_{\ap\in\Pi}$ 
are needed for the description of the Gorenstein highest weight vector varieties, see~\cite[3.7]{p88}, 
\cite[Remark~1.5]{ja99}, or 
for computing cohomology of invertible sheaves on $G/P_\ap$, where 
$P_\ap$ is the maximal parabolic subgroup for $\ap$, see~\cite[4.6]{akh95}.

Let $W_\ap$ be the subgroup of the Weyl group $W$ generated by all {\bf simple} reflections $s_\beta$ 
with $\beta\in\Pi\setminus \{\ap\}$. Then $W_\ap$ is the stabiliser of $\vp_\ap$ in $W$ and also is the 
Weyl group of $\g_\ap(0)$. Write $w_{\ap,0}$ is the longest element in $W_\ap$.  Recall that 
$w_{\ap,0}^2=1$.

\begin{lm}    
\label{prop:edge-labels}
One has $q_\ap\vp_\ap=\vert\mathcal R_\ap\vert$ and $q_\ap\in \BN$.
\end{lm}
\begin{proof}
We have $2\rho=|\Delta^+_\ap(0)|+|\mathcal R_\ap|$ and $(\mu,\vp_\ap)=0$ for
any $\mu\in \Delta^+_\ap(0)$. Hence $(2\rho,\vp_\ap)=(|\mathcal R_\ap|,\vp_\ap)$. Moreover,
$s_\beta(\mathcal R_\ap)=\mathcal R_\ap$ for any $\beta\in\Pi\setminus \{\ap\}$. Therefore,
$|\mathcal R_\ap|$ is proportional to $\vp_\ap$. Clearly, $q_\ap=(|\mathcal R_\ap|,\ap^\vee)$ is an integer.
\end{proof}

\begin{thm}    \label{thm:otmetki}
\quad {\sf 1$^o$.} \ For any $\ap\in\Pi$, we have
\begin{itemize}
\item[\sf (i)] \  $q_\ap\le h$; moreover, $q_\ap=h$ if and only if $\vp_\ap$ is minuscule;
\item[\sf (ii)] \ $q_\ap\ge \rk\g+1$ and this minimum is attained for some $\ap$.
\end{itemize}

{\sf 2$^o$}. \ For any $\ap\in\Pi_l$, one has
$q_\ap\le h^*$; moreover, $q_\ap=h^*$ if and only if $\vp_\ap$ is cominuscule.

{\sf 3$^o$}. \ Suppose that $\theta$ is fundamental and $\widehat\ap\in\Pi$ is such that 
$(\theta,\widehat\ap)\ne 0$. Then $\widehat\ap\in\Pi_l$, $d_{\widehat\ap}=2$, and 
$q_{\widehat\ap}=h^*-1$.

{\sf 4$^o$}. \ If $\theta\ne \theta_s$ and $(\theta_s,\ap)\ne 0$, then $q_{\ap}=h-1$. 
\end{thm}
\begin{proof}
Part {\sf 1$^o$(i)} and the first half of {\sf (ii)} are proved in \cite[Appendix]{ja99}. For the sake of 
completeness, we provide the full argument. 
\\ \indent Clearly, $W_\ap$ preserves each $\Delta_\ap(i)$ and $w_{\ap,0}$ takes the unique
minimal element of each $\Delta_\ap(i)$, $i>0$, to the unique maximal one.

{\sf 1$^o$}. \  We have
$w_{\ap,0}(\mathcal R_\ap )=\mathcal R_\ap $ and
$w_{\ap,0}(\Delta^+_\ap(0))=-\Delta^+_\ap(0)$.
Hence $\rho+w_{\ap,0}\rho=|\mathcal R_\ap |$ and, for any $\gamma\in \mathcal R_\ap $, we have
\[
  (\rho,\gamma^\vee)+(\rho,w_{\ap,0}(\gamma^\vee))=
  (|\mathcal R_\ap |,\gamma^\vee)=q_\ap(\vp_\ap,\gamma^\vee).
\] 
That is, $\hot(\gamma^\vee)+\hot (w_{\ap,0}(\gamma^\vee))=q_\ap(\vp_\ap,\gamma^\vee)$.
Taking $\gamma=\ap$, one obtains 
\beq    \label{eq:q-ap}
1+\hot (w_{\ap,0}(\ap^\vee))=q_\ap .
\eeq
Since $\hot(\gamma^\vee) \le h-1$ for any $\gamma^\vee\in \Delta^\vee$, we have 
$q_\ap\le h$. Furthermore, $q_\ap=h$ if and only if $w_{\ap,0}(\ap^\vee)$ is the
highest root in $\Delta^\vee$. In this case, the equality
$1=(\vp_\ap,\ap^\vee)=(\vp_\ap,w_{\ap,0}(\ap^\vee))$ implies that
$(\vp_\ap,\gamma^\vee)\le 1$ for any $\gamma\in\Delta^+$, i.e.,
$\vp_\ap$ is minuscule. On the other side, $w_{\ap,0}(\ap^\vee)$
is the co-root of maximal height among the roots $\gamma$ such that
$(\vp_\ap,\gamma^\vee)=1$. Since this set contains the co-root
$\sum_{i=1}^n\ap_i^\vee$, we have $\hot (w_{\ap,0}(\ap^\vee))\ge n=\rk\g$.

The existence of $\ap$ such that $q_\ap=\rk\g+1$ can be checked case-by-case. If $\g$ is of type 
$\GR{D}{n}$ or $\GR{E}{n}$, then the branching node of the Dynkin diagram will do. For {\sf\bfseries{BCFG}}, one takes the unique long simple
root that is adjacent to a short root. For $\GR{A}{n}$, all simple roots yield $q_\ap=\rk\g+1=h(\g)$.

{\sf 2$^o$}. \ If $\ap\in\Pi_l$, then  $\hot(w_{\ap,0}(\ap^\vee))\le \hot(\theta^\vee)=h^*-1$, and the equality occurs if and only if $w_{\ap,0}(\ap)=\theta$. In this case, $\theta\in\Delta_\ap(1)$. Hence 
$\hot_\ap(\theta)=\hot_\ap(\ap)=1$, i.e.,
$\vp_\ap$ is comuniscule.

{\sf 3$^o$}. \ Here $(\theta^\vee,\widehat\ap)=(\theta,{\widehat\ap}^\vee)=1$, hence 
$\widehat\ap\in\Pi_l$.
Then $2=(\theta,\theta^\vee)=d_{\widehat\ap}(\widehat\ap,\theta^\vee)=d_{\widehat\ap}$. Next,
$w_{\widehat\ap,0}(\widehat\ap)$ is the maximal root whose $\widehat\ap$-height equals $1$, i.e.,
$w_{\widehat\ap,0}(\widehat\ap)=\theta-\widehat\ap$. Then 
$\hot(w_{\widehat\ap,0}({\widehat\ap}^\vee))=\hot((\theta-\widehat\ap)^\vee)=
\hot(\theta^\vee-{\widehat\ap}^\vee)=h^*-2$ and it follows from \eqref{eq:q-ap} that 
$q_{\widehat\ap}=h^*-1$.

{\sf 4$^o$}. \  Here $\ap\in\Pi_s$ and $\ap^\vee$ is the unique long simple root in $\Pi^\vee$
such that $(\ap^\vee,\theta_s^\vee)\ne 0$. As in the previous part, the $\ap^\vee$-height
of $\theta_s^\vee$ equals $2$ and $w_{\ap,0}(\ap^\vee)=\theta_s^\vee-\ap^\vee$. Since
$\hot(\theta_s^\vee-\ap^\vee)=h-2$, we obtain $q_\ap=\hot(\theta_s^\vee-\ap^\vee)+1=h-1$.
\end{proof}

\begin{ex}    \label{ex:spisok-q}
For the reader's convenience, we list the numbers $\{q_\ap\mid \ap\in\Pi\}$, $h$, and $h^*$ for all simple $\g$.
The numbering of simple roots follows~\cite[Table\,1]{t41}; in particular,  for  $\GR{E}{6}$, $\GR{E}{7}$, and $\GR{E}{8}$, the numbering is 
\\[.9ex] 
\centerline{
\raisebox{-3.1ex}
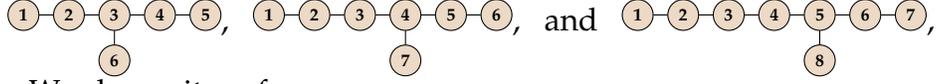
{\begin{tikzpicture}[scale= .55, transform shape]
\tikzstyle{every node}=[circle, draw,fill=brown!30]
\node (a) at (0,0) {\bf 1};
\node (b) at (1.1,0) {\bf 2};
\node (c) at (2.2,0) {\bf 3};
\node (d) at (3.3,0) {\bf 4};
\node (e) at (4.4,0) {\bf 5};
\node (f) at (2.2,-1.1) {\bf 6};
\foreach \from/\to in {a/b, b/c, c/d, d/e, c/f}  \draw[-] (\from) -- (\to);
\end{tikzpicture}}, \ \ 
\raisebox{-3.1ex}{\begin{tikzpicture}[scale= .55, transform shape]
\tikzstyle{every node}=[circle, draw, fill=brown!30]
\node (a) at (0,0) {\bf 1};
\node (b) at (1.1,0) {\bf 2};
\node (c) at (2.2,0) {\bf 3};
\node (d) at (3.3,0) {\bf 4};
\node (e) at (4.4,0) {\bf 5};
\node (f) at (5.5,0) {\bf 6};
\node (g) at (3.3,-1.1) {\bf 7};
\foreach \from/\to in {a/b, b/c, c/d, d/e, e/f, d/g}  \draw[-] (\from) -- (\to);
\end{tikzpicture}}, \ \
and \ \ 
\raisebox{-3.1ex}{\begin{tikzpicture}[scale= .55, transform shape]
\tikzstyle{every node}=[circle, draw, fill=brown!30] 
\node (h) at (-1.1,0) {\bf 1};
\node (a) at (0,0) {\bf 2};
\node (b) at (1.1,0) {\bf 3};
\node (c) at (2.2,0) {\bf 4};
\node (d) at (3.3,0) {\bf 5};
\node (e) at (4.4,0) {\bf 6};
\node (f) at (5.5,0) {\bf 7};
\node (g) at (3.3,-1.1) {\bf 8};
\foreach \from/\to in {h/a, a/b, b/c, c/d, d/e, e/f, d/g}  \draw[-] (\from) -- (\to);
\end{tikzpicture}},}  
\\ respectively.  We also write $q_i$ for $q_{\ap_i}$. 
\begin{enumerate}
\item For $\GR{A}{n}$, one has $q_i=n+1=h=h^*$ for all $i$;
\item For $\GR{B}{n}$, one has $q_i=2n-i$ for $1\le i\le n-1$ and $q_n=2n$; here $h=2n, h^*=2n-1$;
\item For $\GR{C}{n}$, one has $q_i=2n-i+1$ for all $i$; here $h=2n, h^*=n+1$;
\item For $\GR{D}{n}$, one has $q_i=2n-i-1$ for $1\le i\le n-2$ and $q_{n-1}=q_n=2n-2=h$.
\item For $\GR{E}{6}$, $h=h^*=12$ and the numbers $\{q_i\}$ are: \quad \begin{E6}{12}{9}{7}{9}{12}{11}\end{E6}  
\item For $\GR{E}{7}$, $h=h^*=18$ and the numbers $\{q_i\}$ are: \quad \begin{E7}{18}{13}{10}{8}{11}{17}{14}\end{E7}
\item For $\GR{E}{8}$, $h=h^*=30$ and the numbers $\{q_i\}$ are: \quad \begin{E8}{29}{19}{14}{11}{9}{13}{23}{17}\end{E8}
\item For $\GR{F}{4}$, one has $q_1=11=h-1$, $q_2=7$, $q_3=5$, and $q_4=8=h^*-1$;
\item For $\GR{G}{2}$, one has $q_1=5=h-1$ and $q_2=3=h^*-1$.
\end{enumerate}
\end{ex}

\begin{rmk}    \label{rem:refinement-q}
{\sf (1)} Since  $\Delta_\ap(i)$ is the set of weights of a $\g_\ap(0)$-module, we have
$|\Delta_\ap(i)|=q_\ap(i)\vp_\ap$, where $\q_\ap(i)>0$ for $i>0$ and $\sum_{i= 1}^{d_\ap} q_\ap(i)=q_\ap$. This provides a refinement 
of the numbers $\{q_\ap\mid \ap\in\Pi\}$, which we use in Section~\ref{sect:5/2}.
\\ \indent
{\sf (2)} We frequently use the fact that $\#\{\gamma\in\Delta^+\mid (\theta,\gamma)>0\}=2h^*-3$, see~\cite[Prop.\,1]{suter}. 
\end{rmk}

\section{Eigenvalues of Casimir elements associated with $(\BZ,\ap)$-gradings}
\label{sect:5/2}

\noindent
In this section, we fix $\ap\in\Pi$  and work with the $(\BZ,\ap)$-grading 
$\g=\bigoplus_{i\in\BZ}\g_\ap(i)$. Recall that the centre of $\g_\ap(0)$ is one-dimensional (and is 
spanned by $\vp_\ap$ upon the identification of $\te$ and $\te^*$), each $\g_\ap(i)$, $i\ge 1$, is a 
{\bf simple} $\g_\ap(0)$-module, and the set of $\te$-weights of $\g_\ap(i)$ is 
$\Delta_\ap(i)$,~cf. Section~\ref{sect:some-prop}. The {\it height\/} 
of the $(\BZ,\ap)$-grading is $d_\ap=\max_{j\in\BZ}\{j\mid \g_\ap(j)\ne 0\}$. 
The Casimir element in 
$\eus U(\g_\ap(0))$ corresponding to the restriction of $\Phi$ to $\g_\ap(0)$ is denoted by
$\eus C_\ap(0)$. 
Write $\gamma_\ap(i)$ for the eigenvalue of $\eus C_\ap(0)$ on $\g_\ap(i)$. To keep track of the 
length of simple roots, we need $r_\ap=(\theta,\theta)/(\ap,\ap)$. Hence $r_\ap=1$ if and only if 
$\ap\in\Pi_l$. Note that $(\ap,\ap)=1/(h^*r_\ap)$ and  $(\ap,\vp_\ap)=1/(2h^*r_\ap)$.

In the rest of this section, we write $d$ for $d_\ap=\hot_\ap(\theta)$.

\begin{thm}    \label{thm:s-znach-1}
For any $(\BZ{,}\ap)$-grading, we have $\gamma_\ap(1)=\displaystyle\frac{q_\ap}{2h^* r_\ap}$ and 
$\gamma_\ap(d )=1-\displaystyle\frac{d  q_\ap}{2h^* r_\ap}$.
\end{thm}
\begin{proof}
Set $2\rho_\ap(0)=|\Delta^+_\ap(0)|$. Then 
$2\rho=2\rho_\ap(0)+|\mathcal R_\ap|=2\rho_\ap(0)+ q_\ap\vp_\ap$. 
By general principle, if $\eus V_\lb$ is a simple $\g_\ap(0)$-module with the highest weight $\lb$, then
the $\eus C_\ap(0)$-eigenvalue on $\eus V_\lb$ equals $(\lb,\lb+2\rho_\ap(0))$, 
see~\cite[Ch.\,3, Prop.\,2.4]{t41}.
\\ \noindent
\textbullet \quad In our case, $\ap$ is the lowest weight in $\g_\ap(1)$, hence $w_{\ap,0}(\ap)$ is the 
highest weight. Hence 
\begin{multline*}
   \gamma_\ap(1)=(w_{\ap,0}(\ap),w_{\ap,0}(\ap)+2\rho_\ap(0))=(\ap,\ap-2\rho_\ap(0))=
   (\ap,\ap-2\rho)+(\ap, |\mathcal R_\ap|) \\ =(\ap, |\mathcal R_\ap|) 
   =q_\ap(\ap,\vp_\ap)=\frac{q_\ap}{2}(\ap,\ap)=\frac{q_\ap}{2h^*r_\ap} .
\end{multline*}
\textbullet \quad 
Since $\theta$ is the highest weight of the $\g_\ap(0)$-module $\g_\ap(d )$, we obtain
\begin{multline*}
   \gamma_\ap(d )=(\theta,\theta+2\rho_\ap(0))=(\theta,\theta+2\rho)-(\theta,q_\ap\vp_\ap)
   =1-q_\ap(\theta,\vp_\ap) \\
   =1-q_\ap d{\cdot} (\ap,\vp_\ap)=1-\frac{q_\ap d }{2h^*r_\ap} .  \qedhere
\end{multline*}   
\end{proof}
\begin{cl}     \label{cor:2.2}
We have
\begin{itemize}
\item[\sf (i)] \ $d \gamma_\ap(1)+ \gamma_\ap(d )=1$ and hence\/ $1/2d \le \gamma_\ap(1)< 1/d $;
\item[\sf (ii)] \ if\/ $d =1$, i.e., $\vp_\ap$ is cominuscule, then $\gamma_\ap(1)=1/2$;
\item[\sf (iii)] \ if\/ $\theta$ is a multiple of a fundamental weight and $(\widehat\ap,\theta)\ne 0$, then
$\gamma_{\widehat\ap}(1)=(h^*-1)/2h^*$.
\end{itemize}
\end{cl}
\begin{proof}
{\sf (i)} The first equality is clear. Since $\gamma_\ap(d )>0$, one obtains $\gamma_\ap(1)< 1/d $.
On the other hand, any $\eus C_\ap(0)$-eigenvalue in $\bigoplus_{i\ne 0}\g_\ap(i)$ is at most 
$1/2$,~see Prop.~\ref{prop:lms01}.
Hence $\gamma_\ap(d )\le 1/2$ and then $\gamma_\ap(1)\ge 1/2d $.

{\sf (ii)} This follows from (i) with $d=1$.

{\sf (iii)} If $\theta$ is {\bf fundamental}, then $\widehat\ap\in\Pi_l$ and $q_{\widehat\ap}=h^*-1$, see 
Theorem~\ref{thm:otmetki}(3$^o$). Hence the assertion on $\gamma_{\widehat\ap}(1)$. For a 
more general situation in which $\theta$ is a multiple of a fundamental weight, we use the fact that 
$\Delta_{\widehat\ap}(2)=\{\theta\}$ and
$\Delta_{\widehat\ap}(1)=\{\mu\in\Delta^+\mid (\mu,\theta^\vee)=1\}$. Then 
$\#\Delta_{\widehat\ap}(1)=2h^*-4$ (cf. Remark~\ref{rem:refinement-q}(2)) and $\Delta_{\widehat\ap}(1)$ is a union of pairs $\{\mu,\theta-\mu\}$.
Therefore $|\Delta_{\widehat\ap}(1)|=(h^*-2)\theta$ and $|\mathcal R_{\widehat\ap}|=(h^*-1)\theta$. As in 
the proof of Theorem~\ref{thm:s-znach-1},  
$\gamma_{\widehat\ap}(1)=(\widehat\ap,|\mathcal R_{\widehat\ap}|)=
(h^*-1)(\widehat\ap,\theta)=(h^*-1)(\theta,\theta)(\widehat\ap,\theta^\vee)/2=(h^*-1)/2h^*$.
\end{proof}

\begin{rmk}
If $\g$ is classical, then $d \in\{1,2\}$ for all $\ap\in\Pi$. Therefore, Theorem~\ref{thm:s-znach-1}
describes all eigenvalues of all $\eus C_\ap(0)$. 
\end{rmk}

To obtain a general formula for any $\gamma_\ap(i)$, we use the refinement 
$\{q_\ap(i)\}$ of numbers $\{q_\ap\mid \ap\in\Pi\}$, see Remark~\ref{rem:refinement-q}(1). 
Suppose that $1\le k\le d $ and we are going to compute $\gamma_\ap(k)$.
Consider the $\BZ$-graded subalgebra $\g^{[k]}:=\bigoplus_{i\in\BZ}\g_\ap(ki)\subset \g$, i.e., 
$\g^{[k]}(i)=\g_\ap(ki)$. Then
$\g$ and $\g^{[k]}$ share the same $0$-th part and thereby the same Cartan subalgebra $\te\subset \g_\ap(0)$.
 
\begin{lm}      \label{lm:g^k-ss}
$\g^{[k]}$ is semisimple and the root system of\/ $\g^{[k]}$ relative to $\te$  is\/
$\bigsqcup_{i\in\BZ} \Delta_\ap(ki)$. 
\end{lm}
\begin{proof}
The centre of $\g^{[k]}$ (if any) belongs to the centre of $\g_\ap(0)$. As the centre of $\g_\ap(0)$ is
one-dimensional and it acts non-trivially on $\g_\ap(k)$, $\g^{[k]}$ must be semisimple. The rest is clear.
\end{proof}

The passage from $\g$ to $\g^{[k]}$ is a particular case of the general construction outlined in 
Section~\ref{subs:Z-grad-versus} (a passage from $\g$ to $\q$). Because this time we begin with a
$(\BZ,\ap)$-grading, it is possible to say more on the relevant details and the factor $F$. As a result, we end up with an explicit formula for $\gamma_\ap(k)$.
Each graded part $\g^{[k]}(i)=\g_\ap(ki)$ of $\g^{[k]}$ is a simple $\g_\ap(0)$-module. Therefore, the $\BZ$-grading 
of $\g^{[k]}$ is given by a simple root of $\g^{[k]}$. Clearly, this root, say $\beta$, is just the unique 
minimal root in $\Delta_\ap(k)$. Although $\g^{[k]}$ is not necessarily simple, one can write
$\g^{[k]}=\ka\dotplus\es$, where $\ka$ is {\bf simple}, $\es$ is semisimple, and $\beta$ is a simple root of $\ka$.
In this case, the whole of $\es$ lies in $\g_\ap(0)$. Therefore $\g_\ap(0)=\ka_\beta(0)\dotplus\es$ and
$\ka_\beta(i)=\g_\ap(ki)$ for $i\ne 0$. Let $\ov{\vp}_\beta$ be the fundamental weight of $\ka$
(= of $\g^{[k]}$) corresponding to $\beta$.

\begin{prop}    \label{prop:svyaz}
$\vp_\ap=\displaystyle k\frac{(\ap,\ap)}{(\beta,\beta)}\cdot\ov{\vp}_\beta $.
\end{prop}
\begin{proof}
Since either of the weights $\vp_\ap$ and $\ov{\vp}_\beta$ generates the one-dimensional centre of
$\g_\ap(0)$, these are proportional.
By the assumption, $(\vp_\ap, \ap^\vee)=1$ and $(\ov{\vp}_\beta, \beta^\vee)=1$. On the other 
hand, since $\beta$ is a root in $\Delta_\ap(k)$, we have
$(\vp_\ap,\beta^\vee)=k(\vp_\ap,\ap){\cdot} \frac{2}{(\beta,\beta)}=k{\cdot} \frac{(\ap,\ap)}{(\beta,\beta)}$.
Hence $\vp_\ap/\ov{\vp}_\beta=k {\cdot}\frac{(\ap,\ap)}{(\beta,\beta)}$.
\end{proof}

\begin{thm}     \label{thm:gen-formula} 
For any $\ap\in\Pi$ and $1\le k\le d $, one has
$\gamma_\ap(k)=\displaystyle \frac{k }{2h^* r_\ap}\sum_{i\ge 1} q_\ap(ki)$. In particular,
$\gamma_\ap(d)=\displaystyle \frac{d q_\ap(d) }{2h^* r_\ap}$.
\end{thm}
\begin{proof}
As above, we consider $\g^{[k]}=\ka\dotplus\es$ and the  simple root $\beta$ of $\ka$ such that
$\ka_\beta(i)=\g_\ap(ki)$. For the $(\BZ{,}\beta)$-grading of the simple algebra $\ka$, we consider the 
same relevant objects as for $(\g,\ap)$. To distinguish them, the former will be marked by `bar'
(cf. $\vp_\ap$ versus $\ov{\vp}_\beta$). This includes $\ov{q}_\beta,\ov{\mathcal R}_\beta, \ov{r}_\beta$,
etc. (see below).

\textbullet \quad Since $\bigsqcup_{i\in\BZ}\Delta_\ap(ki)=\bigsqcup_{i\in\BZ}\Delta_\beta(i)$ is the 
partition of the root system of $(\ka,\te)$ corresponding to $\beta$, we have 
$|\ov{\mathcal R}_\beta|=\sum_{i\ge 1}|\Delta_\ap(ki)|=\ov{q}_\beta\ov{\vp}_\beta$. On the other hand, 
this sum equals $\sum_{i\ge 1}q_\ap(ki)\vp_\ap$. Invoking Proposition~\ref{prop:svyaz},
we obtain 
\[
   \ov{q}_\beta=k \frac{(\ap,\ap)}{(\beta,\beta)}\sum_{i\ge 1}q_\ap(ki) .
\]
Let $\ov{\eus C}_\beta(0)\in\eus U(\ka)$ be Casimir element associated with the Levi subalgebra
$\ka_\beta(0)\subset\ka$. 
It is important to understand that $\ov{\eus C}_\beta(0)$ is defined via the use of the Killing form 
$\Phi_\ka$ on $\ka$.
Let $\ov{\gamma}_\beta(i)$ denote the eigenvalue of $\ov{\eus C}_\beta(0)$ on $\ka_\beta(i)$.
Set $\ov{r}_\beta=(\ov{\theta},\ov{\theta})/(\beta,\beta)$, where $\ov{\theta}$ is the highest root of $\ka$.
By Theorem~\ref{thm:s-znach-1} applied to $\ka$ and $\beta$, we have
$\ov{\gamma}_\beta(1)=\displaystyle \frac{\ov{q}_\beta}{2h^*(\ka) {\cdot}\ov{r}_\beta}$.

\textbullet \quad Our next step is to compare $\gamma_\ap(k)$ and $\ov{\gamma}_\beta(1)$. Since
$\g_\ap(0)=\ka_\beta(0)\dotplus\es$ and $\es$ acts trivially on each $\g_\ap(ki)$, one can safely remove 
from $\eus C_\ap(0)$ the summands corresponding to the dual bases for $\es$, while computing
$\gamma_\ap(ki)$. This "almost" yields
$\ov{\eus C}_\beta(0)$. The only difference is that the dual bases for $\ka$ occurring in two Casimir elements
are defined via the use of different Killing forms ($\Phi$ and $\Phi_\ka$, respectively). Hence the 
eigenvalues of $\eus C_\ap(0)$ and $\ov{\eus C}_\beta(0)$ on all $\g_\ap(ki)$ are proportional.
More precisely, since the eigenvalues are computed via the use of the canonical bilinear form
on $\te^*$ and $\te^*_\ka$, respectively, the 
transition factor equals the ratio of these two canonical forms.
By Proposition~\ref{prop:transition}(i), this factor equals $T=\displaystyle \frac{h^*(\ka)}{h^*{\cdot} \ind (\ka\hookrightarrow \g)}$.
Gathering together previous formulae, we obtain
\beq \label{eq:polovina}
  \gamma_\ap(k) =T{\cdot}\ov{\gamma}_\beta(1)=
  \frac{h^*(\ka)}{h^*{\cdot} \ind (\ka\hookrightarrow \g)}{\cdot}
  \frac{\ov{q}_\beta}{2h^*(\ka) {\cdot}\ov{r}_\beta}=
  \frac{k{\cdot}(\ap,\ap)\sum_{i\ge 1}q_\ap(ki)}{2h^*{\cdot} \ind(\ka\hookrightarrow \g){\cdot}\ov{r}_\beta {\cdot}(\beta,\beta)}.
\eeq
Proposition~\ref{prop:transition}(ii) says that $\ind (\ka\hookrightarrow \g)=\displaystyle (\theta,\theta)/(\ov{\theta},\ov{\theta})$. Hence 
$\ind (\ka\hookrightarrow \g){\cdot}\ov{r}_\beta{\cdot} (\beta,\beta)=
(\theta,\theta)$, and one simplifies  Eq.~\eqref{eq:polovina}  to 
\[
    \frac{k{\cdot}(\ap,\ap)\sum_{i\ge 1}q_\ap(ki)}{2h^* {\cdot}(\theta,\theta)}=
     \frac{k}{2h^*r_\ap}\sum_{i\ge 1}q_\ap(ki) ,
\]
as required.
\end{proof}

\begin{cl}    \label{cor:sravnenie-2}
For any $\ap\in\Pi$, one has
$d \bigl(q_\ap+q_\ap(d )\bigr)=2h^*r_\ap$ 
and $2h^*r_\ap/d \in \BN$. 
In particular, if $d =2$, then $q_\ap+q_\ap(2)=h^*r_\ap$. 
\end{cl}
\begin{proof}
Theorems~\ref{thm:s-znach-1} and \ref{thm:gen-formula} provide two different formulae for $\gamma_\ap(d )$, 
which yields everything. 
\end{proof}

To apply Theorem~\ref{thm:gen-formula}, one has to know the integers $\{q_\ap(j)\mid 1\le j\le d \}$.
Corollary~\ref{cor:sravnenie-2} allows us to compute $q_\ap(d)$ and thereby settles the problem for $d =2$. For $d>2$, there are some relations between $\{q_\ap(i)\mid i=1,\dots,d\}$, which allows us to solve this problem.

\begin{prop}   \label{prop:simmetri-d_ap}
If $d \ge 2$ and $1\le i\le d-1$, then $q_\ap(i)=q_\ap(d-i)$. 
\end{prop}
\begin{proof}
Consider $\g^{[d]}=\g_\ap(-d)\oplus\g_\ap(0)\oplus\g_\ap(d)$. Then 
$\g^{[d]}$ is the fixed point subalgebra of an automorphism $\psi\in \text{Int}(\g)$ of order $d$. If 
$\zeta =\sqrt[d]1$ is primitive and $1\le i\le d-1$, then the eigenspace of $\psi$ corresponding to 
$\zeta^i$ is
$\g_i:=\g_\ap(i)\oplus \g_\ap(i-d)$.  Since $\g^{[d]}$ is semisimple (Lemma~\ref{lm:g^k-ss}), 
the sum of weights of the $\g^{[d]}$-module $\g_i$ equals $0$. That is, 
\[
    |\Delta_\ap(i)|+ |\Delta_\ap(i-d)|=(q_\ap(i)-q_\ap(d-i))\vp_\ap=0 .  \qedhere
\]
\end{proof}
\begin{rmk}
If $d=3$, then $q_\ap(1)=q_\ap(2)$. That is, Corollary~\ref{cor:sravnenie-2} and 
Proposition~\ref{prop:simmetri-d_ap} are sufficient for computing the numbers $\{q_\ap(j)\}$.
For $d\ge 4$, one can also consider all $\g^{[k]}$ with $k> d/2$, which yields more relations. For 
instance, if $d=4$ and $k=3$, then one get the relation $q_\ap(4)+q_\ap(1)=q_\ap(2)$. All these extra 
relations are sufficient for leisure calculations of all $\{q_\ap(j)\}$. Note that the maximal possible value
$d=6$ is attained only for $\GR{E}{8}$ (once).
\end{rmk}

For future use, we record the following by-product of the above theory.
\begin{prop}      \label{prop:2gamma>1gamma}
For any $\ap\in\Pi$, we have $2\gamma_\ap(1)>\gamma_\ap(2)$. Moreover, if $d$ is odd, then
$\gamma_\ap(1)>\gamma_\ap(2)$.
\end{prop}
\begin{proof}
We have $\gamma_\ap(1)=\displaystyle \frac{\sum_{i\ge 1} q_\ap(i)}{2h^* r_\ap}=\frac{q_\ap}{2h^* r_\ap}$ 
and
$\gamma_\ap(2)=\displaystyle \frac{2\sum_{i\ge 1} q_\ap(2i)}{2h^* r_\ap}$, which yields the first inequality.
For $d$ odd, it follows from Proposition~\ref{prop:simmetri-d_ap} that
$2\sum_{i\ge 1} q_\ap(2i)=\sum_{i=1}^{d-1} q_\ap(i)< q_\ap$.
\end{proof}

\begin{ex}    \label{ex:inequal}
If $d=d_\ap$ is even, then it can happen that $2\gamma_\ap(1)> \gamma_\ap(2) > \gamma_\ap(1)$.
For instance, look up $(\GR{E}{8},\ap_5)$ or $(\GR{E}{8},\ap_6)$ or $(\GR{F}{4},\ap_2)$ in tables in
Appendix~\ref{sect:App}.
\end{ex}
Another interesting relation is

\begin{prop}     \label{prop:k>d/2}
If $k> d/2$ and $\g^{[k]}=\ka\dotplus\es$ as above, then \ 
$\displaystyle  \frac{q_\ap(k)}{\gamma_\ap(k)}=\frac{2h^*}{k}{\cdot}\frac{(\beta,\beta)}{(\ap,\ap)}{\cdot}\ind (\ka\hookrightarrow \g)$.
In particular, for $k=d$, one obtains \ $\displaystyle \frac{q_\ap(d)}{\gamma_\ap(d)}=\frac{2h^*r_\ap}{d}$.
\end{prop}
\begin{proof}
{\sf 1)} \ If $k>d/2$, then $\g^{[k]}=\g_\ap(-k)\oplus\g_\ap(0)\oplus\g_\ap(k)$ has only three summands 
and $\g_\ap(k)$ is commutative. That is, $\ov{\vp}_\beta$ is cominuscule and 
$|\Delta_\ap(k)|=h^*(\ka){\cdot}\ov{\vp}_\beta$. Hence the eigenvalue of $\ov{\eus C}_\beta(0)$ on 
$\ka_\beta(1)=\g_\ap(k)$ equals $1/2$, see Corollary~\ref{cor:2.2}. Using the 
transition factor $T=\displaystyle \frac{h^*(\ka)}{h^*{\cdot} \ind (\ka\hookrightarrow \g)}$ (cf. 
Theorem~\ref{thm:gen-formula}), we obtain
\beq    \label{eq:gamma(k)}
        \gamma_\ap(k)=\displaystyle \frac{h^*(\ka)}{2h^*{\cdot} \ind (\ka\hookrightarrow \g)} .
\eeq 
On the other hand, $|\Delta_\ap(k)|=q_\ap(k)\vp_\ap=h^*(\ka)\ov{\vp}_\beta$.
Hence $h^*(\ka)\ov{\vp}_\beta=q_\ap(k){\cdot}\displaystyle k\frac{(\ap,\ap)}{(\beta,\beta)}\cdot\ov{\vp}_\beta$ and 
\beq   \label{eq:q(k)}
q_\ap(k)= \frac{h^*(\ka)}{k}{\cdot} \frac{(\beta,\beta)}{(\ap,\ap)} .
\eeq
Combining Eq.~\eqref{eq:gamma(k)} and~\eqref{eq:q(k)} yields the first assertion.

{\sf 2)} \ If $k=d$, then $\beta$ is the minimal root in $\Delta_\ap(d)$, which is $W_\ap$-conjugate to 
$\theta$, the
maximal root in $\Delta_\ap(d)$. Hence $\beta$ is long and $\ov{\theta}=\theta$. Therefore,
$(\beta,\beta)/(\ap,\ap)=r_\ap$ and
$\ind (\h\hookrightarrow \g)=1$, cf. Proposition~\ref{prop:transition}.
\end{proof}
{\bf Remark.} Comparing Proposition~\ref{prop:k>d/2} and Corollary~\ref{cor:sravnenie-2}, we see that
$\displaystyle \frac{q_\ap(d)}{\gamma_\ap(d)}=q_\ap+q_\ap(d)$ is an integer.

\begin{ex}    \label{ex:primer-E8}
{\sf (1)} \ Consider the $(\BZ{,}\ap_2)$-grading of $\GR{E}{8}$. Here $d=3$ and $q_2=19$ (see
Example~\ref{ex:spisok-q}). By Theorem~\ref{thm:s-znach-1},
$\gamma_{\ap_2}(1)=19/60$ and $\gamma_{\ap_2}(3)=3/60$. Then Corollary~\ref{cor:sravnenie-2}
shows that $q_2(3)=(60/3)-19=1$. Hence $q_2(1)=q_2(2)=9$. Now, using Theorem~\ref{thm:gen-formula}, we compute that $\gamma_{\ap_2}(2)=18/60$.

{\sf (2)} \  Take the $(\BZ{,}\ap_2)$-grading of $\GR{F}{4}$. Here $r_{\ap_2}=2$, $d=4$, and 
$q_2=7$. By Theorem~\ref{thm:s-znach-1}, $\gamma_{\ap_2}(1)=7/36$ and 
$\gamma_{\ap_2}(4)=1-(4{\cdot}7/36)=8/36$. Then Corollary~\ref{cor:sravnenie-2} shows that 
$q_2(4)=(2{\cdot}9{\cdot}2/4)-7=2$. Since $q_2(1)=q_2(3)$ and $q_2(4)+q_2(1)=q_2(2)$, one computes the remaining $q_2(j)$'s. Finally, Theorem~\ref{thm:gen-formula} implies that 
$\gamma_{\ap_2}(2)=10/36$ and $\gamma_{\ap_2}(3)=3/36$.
\end{ex}
The complete calculations of the eigenvalues $\{\gamma_\ap(i)\}$ and integers $\{q_\ap(i)\}$ for
all $(\BZ{,}\ap)$-gradings are gathered in Appendix~\ref{sect:App}.

\section{Eigenvalues of $\eus C_\ap(0)$ in $\bigwedge^k\g_\ap(1)$ and abelian subspaces of 
$\g_\ap(1)$}
\label{sect:3}

\noindent
In this section, we relate eigenvalues of $\eus C_\ap(0)$ to dimensions of abelian subspaces
(=\, commutative subalgebras) of $\g_\ap(1)$. The key role is played by the inequality 
$\gamma_\ap(2)< 2\gamma_\ap(1)$, see  Prop.~\ref{prop:2gamma>1gamma}. As an application of our
theory, we prove that if $d_\ap>1$ and $\ah\subset\g_\ap(1)$ is an abelian subspace, i.e., $[\ah,\ah]=0$,
then $\dim\ah\le  \dim\g_\ap(1)/2$. Our results up to Proposition~\ref{prop:mult-free} are parallel to 
results of~\cite[Sect.\,4]{jlms01} that concern the case of $\BZ_m$-gradings. Furthermore, most 
proofs therein can readily be adapted to the $\BZ$-graded setting.
For this reason, we omit some details.

Let us recall some basic facts on the complexes $(\bigwedge^\bullet\!\g, \partial)$ 
and $(\bigwedge^\bullet\!\g, \textsl{d})$.
We identify $\g$ with $\g^*$, using $\Phi$, and consider
$\bigwedge^\bullet\g$ with the usual differentials:
\[  \textstyle
\text{$\textsl{d}:\bigwedge^l\g \to \bigwedge^{l+1}\g$ \quad and \quad
$\partial:\bigwedge^l\g \to \bigwedge^{l-1}\g.$}
\] 
Here
\[
\partial(x_1\wedge\ldots\wedge x_l):=\sum_{i<j}(-1)^{i+j-1}
[x_i,x_j]\wedge x_1\wedge\ldots \hat{x}_i\dots \hat{x}_j\ldots\wedge x_l 
\]
for $l\ge 2$ and $\partial(x_1)=0$. In particular, 
$\partial(x_1\wedge x_2)=[x_1,x_2]$. 
\\
We regard $\Phi$ as having been extended, in the usual way, via
determinants, from $\g$ to $\bigwedge^\bullet\!\g$. 
More precisely, denoting the extension of $\Phi$ to $\bigwedge^l\g$ by
$\Phi^{(l)}$, we have
\[
\Phi^{(l)}(x_1\wedge\ldots\wedge x_l,y_1\wedge\ldots\wedge y_l):=
\det \| \Phi(x_i,y_j)\| \ .
\] 
Then $\textsl{d}=-\partial^t$.
For $x\in\g$, let $\esi(x)$ be the exterior product operator and 
$i(x)$ the interior product operator in $\bigwedge\g$. That is, 
\[
\esi(x){\cdot}x_1\wedge\ldots\wedge x_l:=x\wedge x_1\ldots\wedge x_l,
\]
\[
i(x){\cdot}x_1\wedge\ldots\wedge x_l:=\sum_{i=1}^l(-1)^{i-1}
\Phi(x,x_i)x_1\wedge\ldots \hat{x_i}\ldots \wedge x_l \ .
\]
Let $\vartheta$ denote the natural extension of the adjoint representation
of $\g$ to $\bigwedge^\bullet\!\g$: 
\[
\vartheta(x){\cdot}x_1\wedge\ldots\wedge x_l:=\sum_{i=1}^l
x_1\wedge\ldots \wedge [x,x_i]\wedge\ldots \wedge x_l \ .
\]
These operators satisfy the following relations for all $x\in\g$:
\beq     \label{eq:3shtuki} 
[\textsl{d},\vartheta(x)]=0, \quad  
[ \partial , \vartheta (x) ]=0 , \quad
\esi(x)\partial+\partial\esi(x)=\vartheta(x) .
\eeq
Let $e_1,\dots,e_N$ be a basis for $\g$ and $e'_1,\dots,e'_N$ the dual basis.
After Koszul \cite[3.4]{kos}, it is known that
\beq   \label{eq:d}
\textsl{d}=\frac{1}{2}\sum_{i=1}^N \esi(e'_i)\vartheta(e_i) \ .
\eeq
Combining Eq.~\eqref{eq:3shtuki} and \eqref{eq:d} yields
\[
2(\textsl{d}\partial+\partial \textsl{d})=\sum_{i=1}^N\vartheta(e'_i)\vartheta(e_i)=\vartheta(\eus C) ,
\]
where $\eus C$ is the Casimir element for $\g$. In the general $\BZ$-graded situation, we choose a basis
$\mathbb B=(e_1,\dots,e_N)$
compatible with grading, which means that $\mathbb B\cap\g(i)$ is a basis for $\g(i)$ for each $i$. 
Let $\mathbb B'=(e'_1,\dots,e'_N)$ be the dual basis. Since $\g(i)^*\simeq \g(-i)$, we have 
$(\mathbb B\cap\g(i))'=\mathbb B'\cap\g(-i)$ is a basis for $\g(-i)$. Any compatible basis yields a splitting of the differential: 
$\textsl{d}=\sum_{i\in\BZ}\textsl{d}_i$, where
\[
    \textsl{d}_i= \frac{1}{2}\sum_{j:\ e_j\in \g(i)} \esi(e'_j)\vartheta(e_j) .
\]
Note that $\textsl{d}_i(\g(j))\subset \begin{cases} \g(i)\otimes\g(j-i), & i\ne j/2 \\
\bigwedge^2\!\g(i), & i=j/2 \end{cases} \ \subset \bigwedge^2\!\g$. In particular, $\textsl{d}_1(\g(2))
\subset \bigwedge^2\!\g(1)$. 
The key technical result for $\g(1)$ and $\textsl{d}_1$ is
\begin{prop}    \label{prop:d_1}
{\sf (i)} \ Let $(e_1,\dots,e_s)$ be a basis for $\g(1)$ and $(e'_1,\dots,e'_s)$ the dual basis for $\g(-1)$.
For any $y,z\in \g(1)$, we have 
\[
     \sum_{i=1}^s [e_i,y]\wedge [e'_i,z]=-\textsl{d}_1([y,z]) .
\]
{\sf (ii)} \ For any $x\in\g(2)$ and $u,v\in\g(-1)$, we have 
$\Phi^{(2)}(\textsl{d}_1(x), u\wedge v)=-\Phi(x,\partial(u\wedge v))$.
\end{prop}
\begin{proof}
The argument is essentially the same as in the proof of the similar results for $\BZ_m$-gradings, 
see Proposition 4.1 and Eq.~(4.4) in~\cite{jlms01}. 
\end{proof}

The above assertion holds for any  $\BZ$-grading. Below, we again assume that $\g(i)=\g_\ap(i)$ for
some $\ap\in\Pi$ and consider related eigenvalues of $\eus C_\ap(0)$.  

\begin{prop}   \label{prop:wedge-2}
Any  $\eus C_\ap(0)$-eigenvalue in $\bigwedge^2\!\g_\ap(1)$ is at most $2\gamma_\ap(1)$.
The eigenspace for $2\gamma_\ap(1)$ is spanned by the bi-vectors $y\wedge z$ such that $[y,z]=0$.
\end{prop}
\begin{proof}
By Proposition~\ref{prop:d_1}, we have
\[
   \eus C_\ap(0)(y\wedge z)=(\eus C_\ap(0)y)\wedge z+y\wedge (\eus C_\ap(0)z)+
   2\sum_{i=1}^s [e_i,y]\wedge [e'_i,z]=2\gamma_\ap(1){\cdot}y\wedge z-2\textsl{d}_1([y,z]) .
\]
By~\cite[Prop.\,4]{ko65} and \cite[Theorem\,3.2]{jlms01}, the maximal eigenvalue of $\eus C_\ap(0)$ 
in $\bigwedge^k\!\g_\ap(1)$  is attained on decomposable
polyvectors. So, we may assume that $y\wedge z$ is a $\eus C_\ap(0)$-eigenvector.

\textbullet \quad Assume that $[y,z]\ne 0$. Since $[y,z]=\partial (y\wedge z)$ and $\partial$ is $\g$-equivariant, the $\eus C_\ap(0)$-eigenvalues of $y\wedge z$ and $[y,z]$ are equal. As
$[y,z]\in\g_\ap(2)$, its eigenvalue equals $\gamma_\ap(2)$. We also know that $\gamma_\ap(2) <2\gamma_\ap(1)$, see Proposition~\ref{prop:2gamma>1gamma}.

\textbullet \quad If $[y,z]= 0$, then $\eus C_\ap(0)(y\wedge z)=2\gamma_\ap(1){\cdot}y\wedge z$.
\end{proof}

Actually, there is a more precise result for $\bigwedge^2\!\g_\ap(1)$.
\begin{thm}   \label{thm:wedge-2}
Let $\eus A_2=\text{\rm span}\{ y\wedge z\in \bigwedge^2\!\g_\ap(1)\mid [y,z]=0\}$. Then
\begin{itemize}
\item[\sf (i)] \  $\eus A_2=\Ker(\partial\vert_{\bigwedge^2\!\g_\ap(1)})$;
\item[\sf (ii)] \  $\bigwedge^2\!\g_\ap(1)=\eus A_2\oplus \textsl{d}_1(\g_\ap(2))$; hence
  $\eus C_\ap(0)$ has at most two eigenvalues in $\bigwedge^2\!\g_\ap(1)$.
\end{itemize}
\end{thm}
\begin{proof}
{\sf (i)} \ If $u=\sum_i y_i\wedge z_i$, then 
\[
   \eus C_\ap(0)u=2\gamma_\ap(1){\cdot}u-
2\sum_i\textsl{d}_1([y_i,z_i])=2\gamma_\ap(1){\cdot}u-2\textsl{d}_1\partial(u).
\]
Therefore, if $\partial(u)=0$, then $u\in \eus A_2$ in view of Proposition~\ref{prop:wedge-2}. Hence
$\Ker(\partial\vert_{\bigwedge^2\!\g_\ap(1)})\subset \eus A_2$, and the opposite inclusion is obvious.

{\sf (ii)} \ Since $\textsl{d}_1: \g_\ap(2)\to \bigwedge^2\!\g_\ap(1)$ is $\g_\ap(0)$-equivariant,
the eignevalue of $\eus C_\ap(0)$ in $\textsl{d}_1(\g_\ap(2))$, which is $\gamma_\ap(2)$, is strictly less than
$2\gamma_\ap(1)$. Therefore, we have $\bigwedge^2\!\g_\ap(1)=\eus A_2\oplus \textsl{d}_1(\g_\ap(2))\oplus\mathbb U$
with some $\mathbb U$.
Then it follows from {\sf (i)} that
$\partial(\bigwedge^2\!\g_\ap(1))=\partial(\textsl{d}_1(\g_\ap(2))\oplus\partial\,\mathbb U$ 
and  $\dim \partial(\textsl{d}_1(\g_\ap(2))=\dim (\textsl{d}_1(\g_\ap(2))$. Using Proposition~\ref{prop:d_1}(ii),
we see that $\dim (\textsl{d}_1(\g_\ap(2))=\dim \partial(\bigwedge^2\!\g_\ap(-1))=
\dim (\bigwedge^2\!\g_\ap(1))$. Hence $\partial\,\mathbb U=0$ and then $\mathbb U=0$.
\end{proof}

\begin{rmk}   \label{rem:pro-A2}
It follows from Theorem~\ref{thm:wedge-2} that 
\[   \textstyle
\bigwedge^2\!\g_\ap(1)\simeq \g_\ap(2) \Longleftrightarrow\ \eus A_2=0\ \Longleftrightarrow 
\g_\ap(1) \text{ has no 2-dim abelian subspaces}.
\] 
This possibility does materialise for the $(\BZ,\ap)$-gradings $(\GR{B}{n}, \ap_n)$ and $(\GR{G}{2},\ap_1)$.
\end{rmk}

The following is a natural generalisation of Proposition~\ref{prop:wedge-2}. 
\begin{thm}     \label{thm:main3}
For any $k\ge 1$, the maximal eigenvalue of  $\eus C_\ap(0)$ in $\bigwedge^k\!\g_\ap(1)$ is at most $k\gamma_\ap(1)$.
This bound is attained if and only if $\g_\ap(1)$ contains a $k$-dimensional commutative subalgebra.
In that case, the eigenspace belonging to $k\gamma_\ap(1)$ is spanned by the polyvectors 
$\bigwedge^k\!\ah$, where $\ah\subset\g_\ap(1)$ is $k$-dimensional and $[\ah,\ah]=0$.
\end{thm}
\begin{proof}
The  proof of the similar result related to $\BZ_m$-gradings goes through {\sl mutatis mutandis\/}
(cf. \cite[Theorem\,4.4]{jlms01}). Following ideas of Kostant~\cite{ko65}, one has to use a compact real 
form of $\g$ and a related Hermitian inner product on $\g$.
\end{proof}

Let $\eus A_k$ be the subspace of $\bigwedge^k\!\g_\ap(1)$ spanned by the ``$k$-dimensional 
commutative subalgebras'', i.e.,
\[
    \eus A_k=\text{\rm span}\{ y_1\wedge\ldots\wedge y_k\mid y_i\in\g_\ap(1) \ \& \ [y_i,y_j]=0
    \ \forall i,j\} .
\]
Then $\eus A=\oplus_{k\ge 1}\eus A_k$ is a $\g_\ap(0)$-submodule of $\bigwedge^\bullet\!\g_\ap(1)$.
\begin{prop}       \label{prop:mult-free}
$\eus A$ is a multiplicity-free\/ $\g_\ap(0)$-module.
\end{prop}
\begin{proof}
Set $\be_\ap(0)=\be\cap\g_\ap(0)$. If $\lb$ is a highest weight in $\eus A$ w.r.t. $\be_\ap(0)$, then 
there is a $\be_\ap(0)$-stable abelian subspace $\ah\subset \g_\ap(1)$ such that the set of $\te$-roots 
of $\ah$ is $\Delta_\ah=\{\mu_1,\dots,\mu_l\}$ and $\lb=\sum_{i=1}^l \mu_i$; and vice versa. The
$\be_\ap(0)$-invariance of $\ah$ means that $\Delta_\ah$ is an {\it upper $\Delta^+_\ap(0)$-ideal\/} in 
the sense that if $\mu_i+\eta\in\Delta^+$ for some $\eta\in\Delta^+_\ap(0)$, then 
$\mu_i+\eta\in\Delta_\ah$.
\\ \indent
Assume that there are two $\be_\ap(0)$-stable commutative subalgebras of $\g_\ap(1)$
whose dominant weights coincide. That is, $\ah\sim \Delta_\ah=\{\mu_1,\dots,\mu_l\}$,
$\ah'\sim \Delta_{\ah'}=\{\nu_1,\dots,\nu_m\}$, and $\sum \mu_i=\sum \nu_j$. Removing the common 
elements of these two sets, we have
\[
    |\Delta_\ah\setminus \Delta_{\ah'}|= |\Delta_{\ah'}\setminus \Delta_{\ah}|
\]
Hence $(|\Delta_\ah\setminus \Delta_{\ah'}|, |\Delta_{\ah'}\setminus \Delta_{\ah}|)>0$ and there are
$\mu_i\in\Delta_\ah\setminus \Delta_{\ah'} , \nu_j\in \Delta_{\ah'}\setminus \Delta_{\ah}$ such that
$(\mu_i,\nu_j)>0$. Then $\mu_i-\nu_j\in \Delta_\ap(0)$, since $\hot_\ap(\mu_i)=\hot_\ap(\nu_j)$.
If, for instance, $\nu_j-\mu_i$ is positive, then $\nu_j\in \Delta_{\ah}$. A contradiction!
\end{proof}

If $d_\ap=1$, then $\g_\ap(1)$ is commutative. Conversely, if $d_\ap\ge 2$, then $[\g_\ap(1),\g_\ap(1)]=\g_\ap(2)$,
i.e., $\g_\ap(1)$ is {\bf not} commutative. From our theory, we derive a more precise assertion.

\begin{thm}    \label{thm:comm-dim}
For any $(\BZ{,}\ap)$-grading of $\g$, one has
\begin{itemize}
\item either\/  $\g_\ap(1)$ is abelian (which happens if and only if $d_\ap=1$);
\item or\/ $\dim\ah\le \frac{1}{2}\dim\g_\ap(1)$ for any abelian subspace $\ah\subset\g_\ap(1)$.
\end{itemize}
\end{thm}
\begin{proof} It suffices to prove that if $\g_\ap(1)$ contains an abelian subspace $\ah$ such that
$\dim\ah> \frac{1}{2}\dim\g_\ap(1)$, then $\g_\ap(1)$ is abelian.

Set $m=\dim\g_\ap(1)$ and let $\delta_\ap(k)$ be the maximal eigenvalue of $\eus C_\ap(0)$ on 
$\bigwedge^k\!\g_\ap(1)$. Then
$\delta_\ap(1)=\gamma_\ap(1)$ and we have proved that $\delta_\ap(k)\le k\gamma_\ap(1)$ for any $k$.
Note that $\delta_\ap(m)$ is just the eigenvalue on the $1$-dimensional module
$\bigwedge^m\!\g_\ap(1)$.

Write $\g_\ap(0)=\es\oplus \langle h_\ap\rangle$, where $\es$ is semisimple and $h_\ap$ is the element of the centre that has the eigenvalue $k$  on $\g_\ap(k)$. Then $h_\ap$ also has 
eigenvalue $k$  on $\bigwedge^k\!\g_\ap(1)$. We have
\beq     \label{eq:casimir-summa}
  \eus C_\ap(0)=\eus  C_\es+ h_\ap h'_\ap=\eus  C_\es+ h^2_\ap /(h_\ap,h_\ap) .
\eeq
Since $\bigwedge^k\!\g_\ap(1)$ and $\bigwedge^{m-k}\!\g_\ap(1)$ are isomorphic as $\es$-modules,
their $\eus C_\es$-eigenvalues coincide; the difference in $\eus C_\ap(0)$-eigenvalues comes the presence of the last summand.

An easy observation is that if the summand $h^2_\ap /(h_\ap,h_\ap)$ has the eigenvalue $\chi$ on 
$\g_\ap(1)$, then its eigenvalue on $\bigwedge^k\!\g_\ap(1)$ is $k^2\chi$.

Assume that $k< m/2$ and $\g_\ap(1)$ has a commutative subalgebra of dimension 
$m-k$. Then it has a $k$-dimensional commutative subalgebra, too. Hence
$\delta_\ap(k)=k\gamma_\ap(1)$ and $\delta_\ap(m-k)=(m-k)\gamma_\ap(1)$. 
Let $F_i$ be the maximal eigenvalue of $\eus  C_\es$ on $\bigwedge^i\!\g_\ap(1)$.  Then
$F_i=F_{m-i}$ and, using
the decomposition in~\eqref{eq:casimir-summa}, we can write \\
\centerline{
$\begin{cases}   \delta_\ap(k)& =F_k+k^2\chi=k\gamma_\ap(1), \\
\delta_\ap(m-k)& =F_{k}+(m-k)^2\chi=(m-k)\gamma_\ap(1) . \end{cases}$}
\\[.6ex]
Taking the difference yields $m(m-2k)\chi=(m-2k)\gamma_\ap(1)$. Note also that $F_0=F_m=0$, since
$\bigwedge^m\!\g_\ap(1)$ is a trivial $\es$-module. Hence
$\delta_\ap(m)=m^2\chi=m\gamma_\ap(1)$. By Theorem~\ref{thm:main3}, this means that
$\g_\ap(1)$ is commutative.
\end{proof}

\begin{rmk}   \label{rem:pro-alela}
It was recently noticed that, for any nilpotent element $e\in\g$ and the associated Dynkin $\BZ$-grading (so that $e\in\g(2)$),  one has $\dim\ah\le (\dim\g(1))/2$ whenever $\ah\subset\g(1)$ and $[\ah,\ah]=0$, 
see~\cite[Prop.\,3.1]{e-j-k}. In this case, $\dim\g(1)$ is necessarily even.
However, there are $(\BZ{,}\ap)$-gradings of height $\ge 2$ that are not Dynkin gradings, and it can 
also happen that $\dim\g_\ap(1)$ is odd. 
\\ \indent
In fact, conversations with A.G.\,Elashvili on results of~\cite{e-j-k} revived my memory of~\cite{jlms01} 
and triggered my interest to eigenvalues of the Casimir elements related to Levi subalgebras and 
$\BZ$-gradings.
\end{rmk}

\begin{rmk}   \label{rem:ne-vsegda-polovina}
For an arbitrary $\BZ$-grading of $\g$, it can happen that $\g(1)$ is not abelian, but 
$\dim\ah> \frac{1}{2}\dim\g(1)$ for some abelian $\ah\subset\g(1)$. Suppose that a $\BZ$-grading
$\g=\bigoplus_{i\in\BZ}\g(i)$ is given by a function $\ff:\Pi\to\{0,1\}$, i.e., $\g^\ap\subset \g(\ff(\ap))$, cf.
~\cite[Ch.\,3,\,\S3.5]{t41}. Set $\Pi_1=\{\ap\mid \ff(\ap)=1\}$.
Then $\g(1)=\bigoplus_{\ap\in\Pi_1}\eus V(\ap)$, where $\eus V(\ap)$ is a simple $\g(0)$-module with 
{\bf lowest} weight $\ap$. The set of weights of $\eus V(\ap)$ is 
\[
   \{\gamma\in\Delta^+ \mid \hot_\ap(\gamma)=1 \ \& \ \hot_\beta(\gamma)=0 \ \ \forall \beta\in\Pi_1\setminus\{\ap\}\} .
\] 
Take, for instance, $\Pi_1=\{\ap_2,\ap_4\}$ for $\GR{A}{6}$. Then both $\eus V(\ap_2)$ and 
$\eus V(\ap_4)$ are
abelian, of different dimension, but $0\ne [\eus V(\ap_2), \eus V(\ap_4)]=\eus V(\ap_2+\ap_3+\ap_4)\subset \g(2)$. Here
$6=\dim \eus V_{\ap_4}> \frac{1}{2}\dim\g(1)=5$.
\end{rmk}

Below, we elaborate on some numerology related to the numbers $q_\ap(i), \dim\g_\ap(i),
(\vp_\ap,\vp_\ap)$, $\delta_\ap(m)$,
etc. Recall that $m=\dim\g_\ap(1)$.

Let $C=(c_{\ap\beta})_{\ap,\beta\in\Pi}$ be the inverse of the Cartan matrix of $\g$. Then $\vp_\ap=\sum_{\beta\in\Pi} c_{\ap\beta}\beta$. Therefore, $(\vp_\ap,\vp_\ap)=c_{\ap\ap}(\vp_\ap,\ap)=
c_{\ap\ap}(\ap,\ap)/2=c_{\ap\ap}/2h^*r_\ap$.

\begin{prop}    \label{prop:delta-top}
For any $\ap\in\Pi$, we have $\delta_\ap(m)=q_\ap(1)^2 (\vp_\ap,\vp_\ap)$.
\end{prop}
\begin{proof}
The weight of the $1$-dimensional $\g_\ap(0)$-module $\bigwedge^m\!\g_\ap(1)$ is $|\Delta_\ap(1)|=q_\ap(1)\vp_\ap$. Since 
$(\vp_\ap,2\rho)=q_\ap(\vp_\ap,\vp_\ap)$, we obtain
\begin{multline*}
\delta_\ap(m)=(q_\ap(1)\vp_\ap,q_\ap(1)\vp_\ap+2\rho_\ap(0))=
q_\ap(1)(\vp_\ap, |\Delta_\ap(0)|+|\Delta_\ap(1)|) 
\\ =q_\ap(1)(\vp_\ap, 2\rho- \sum_{i\ge 2}|\Delta_\ap(i)|)
=q_\ap(1)(\vp_\ap, 2\rho- (q_\ap-q_\ap(1))\vp_\ap)
\\ =q_\ap(1) (\vp_\ap,\vp_\ap){\cdot}\bigl(q_\ap-(q_\ap-q_\ap(1))\bigr)=
q_\ap(1)^2{\cdot} (\vp_\ap,\vp_\ap) .   \qedhere
\end{multline*}
\end{proof}

\begin{cl}    \label{cor:delta-top-ab}  \leavevmode\par
{\sf (i)} \ If\/ $\vp_\ap$ is cominuscule, then  $(\vp_\ap,\vp_\ap)=m/2(h^*)^2$, $\delta_\ap(m)=m/2$,
and $m=c_{\ap\ap}h^*$.
\\ \indent
{\sf (ii)} \ If\/ $\theta$ is fundamental and $(\widehat\ap,\theta)\ne 0$, then 
$(\vp_{\widehat\ap},\vp_{\widehat\ap})=1/h^*$ and $\delta_{\widehat\ap}(m)=(h^*-2)^2/h^*$.
\end{cl}
\begin{proof}
{\sf (i)} \ Here $r_\ap=1$ and $\g_\ap(1)$ is commutative, hence $q_\ap=q_\ap(1)=h^*$, 
$\gamma_\ap(1)=1/2$, and $\delta_\ap(k)=k/2$ for every $k$. Then  $(h^*)^2(\vp_\ap,\vp_\ap)=m/2$, and we are done.
\\ \indent
{\sf (ii)} \ Here $q_{\widehat\ap}(1)=h^*-2$ (cf. the proof of Corollary~\ref{cor:2.2}(iii)) and 
$\vp_{\widehat\ap}=\theta$, i.e., $(\vp_{\widehat\ap},\vp_{\widehat\ap})=1/h^*$. Note also that here $c_{\widehat\ap\widehat\ap}=\hot_{\widehat\ap}(\theta)=2$.
\end{proof}

\begin{prop}     \label{prop:hot_ap-etc}
For any $\ap\in\Pi$  and $k\in\BN$, we have  
\[
    \hot_\ap(|\Delta_\ap(k)|)=k{\cdot} \dim\g_\ap(k)=c_{\ap\ap} q_\ap(k) .
\]
\end{prop}
\begin{proof} 
Since $(\nu,\vp_\ap)=k(\ap,\vp_\ap)$ for any $\nu\in \Delta_\ap(k)$ and $\dim\g_\ap(k)=\#
\Delta_\ap(k)$, we have  
$(|\Delta_\ap(k)|,\vp_\ap)=k{\cdot} \dim\g_\ap(k)(\ap,\vp_\ap)$. On the other hand, 
\[
  (|\Delta_\ap(k)|,\vp_\ap)=\bigl(\hot_\ap(|\Delta_\ap(k)|)\ap+\dots,\vp_\ap\bigr)=
  \hot_\ap(|\Delta_\ap(k)|){\cdot}(\ap,\vp_\ap) ,
\] 
which gives the first equality.  Likewise, 
\[
  (|\Delta_\ap(k)|,\vp_\ap)=q_\ap(k){\cdot}(\vp_\ap,\vp_\ap)=q_\ap(k){\cdot}c_{\ap\ap}(\ap,\vp_\ap) .  \qedhere
\]
\end{proof}
\begin{cl}
The ratio $\bigl(k{\cdot}\dim\g_\ap(k)\bigr)/q_\ap(k)=c_{\ap\ap}$ does not depend on $k$. In particular, any linear 
relation between the $q_\ap(i)$'s translates into a linear relation between the $\dim\g_\ap(i)$'s. 
\end{cl}
{\bf Example}. By Proposition~\ref{prop:simmetri-d_ap}, one has $q_\ap(d-i)=q_\ap(i)$. Hence 
$(d-i){\cdot}\dim\g_\ap(d-i)=i{\cdot}\dim\g_\ap(i)$.
In particular, $\dim\g_\ap(d-1)=\dim\g_\ap(1)/(d-1)$.

\section{Maximal abelian subspaces and applications}
\label{sect:applic}

\noindent
We say that $\g_\ap(1)$ has an {\it abelian subspace of half-dimension}, if there is 
$\ah\subset\g_\ap(1)$ such that $[\ah,\ah]=0$ and $\dim\ah=\frac{1}{2}\dim \g_\ap(1)$. First we 
discuss some consequences of this property.

\begin{prop}    \label{prop:all-sigma(i)}
Suppose that $\g_\ap(1)$ has an abelian subspace of half-dimension, $m=\dim\g_\ap(1)$, and
$k\le m/2$. Then $\delta_\ap(k)=k\gamma_\ap(1)$ and
\beq    \label{eq:delta1(m-k)}
      \delta_\ap(m-k)=k\gamma_\ap(1)+(m-2k)\frac{q_\ap(1)}{2h^*r_\ap} .
\eeq
\end{prop}
\begin{proof}
The first relation follows from Theorem~\ref{thm:main3}. Next, we know that 
$\delta_\ap(k)=F_k+k^2\chi$, see the proof of Theorem~\ref{thm:comm-dim}. Hence
$F_k=k\gamma_\ap(1)-k^2\chi$. Since $F_k=F_{m-k}$, we obtain
\beq        \label{eq:delta2(m-k)}
  \delta_\ap(m-k)=F_{m-k}+(m-k)^2\chi=F_{k}+(m-k)^2\chi=k\gamma_\ap(1)+m(m-2k)\chi .
\eeq
As $F_0=F_m=0$, we compute $\chi$ and $\delta_\ap(m)$ using the numerology of Section~\ref{sect:3}:
\[
  m^2\chi=\delta_\ap(m)=q_\ap(1)^2{\cdot}(\vp_\ap,\vp_\ap)=
  q_\ap(1)^2{\cdot}c_{\ap\ap}{\cdot}(\ap,\ap)/2=\frac{m q_\ap(1)}{2h^*r_\ap} .
\] 
Here the relation $q_\ap(1){\cdot} c_{\ap\ap}=\dim\g_\ap(1)$ is used, see Prop.~\ref{prop:hot_ap-etc}.  Thus, 
$m\chi=\displaystyle \frac{q_\ap(1)}{2h^*r_\ap}$ and plugging this into Eq.~\eqref{eq:delta2(m-k)}
yields Eq.~\eqref{eq:delta1(m-k)}.
\end{proof}
{\bf Remark.} We obtain here a formula for $\delta_\ap(i)$ for {\bf all} $i\in\{1,\dots,m\}$. More generally, if $\max(\dim\ah)=r\le m/2$, then the same argument yields $\delta_\ap(i)$ for $i\le r$ and $i\ge m-r$.
\begin{cl}    \label{prop:gorka-versus}
Under the above assumptions,
\begin{itemize}
\item[\sf (i)] \ if $q_\ap > 2q_\ap(1)$, then $\max_{i}\{\delta_\ap(i)\}=\delta_\ap(m/2)$ and the sequence
$\{\delta_\ap(i)\}$ is unimodal;
\item[\sf (ii)] \ if $q_\ap= 2q_\ap(1)$, then the sequence $\{\delta_\ap(i)\}$ stabilises after $i=m/2$;
\item[\sf (iii)] \ if $q_\ap < 2q_\ap(1)$, then $\max_{i}\{\delta_\ap(i)\}=\delta_\ap(m)$ and the sequence
$\{\delta_\ap(i)\}$ strictly increases.
\end{itemize}
Furthermore, if $d_\ap\ge 3$, then case {\sf (i)} always occurs.
\end{cl}
\begin{proof} 
The sequence $\{\delta_\ap(i)\}$ clearly increases for $1\le i\le m/2$.
By Theorem~\ref{thm:s-znach-1}, one has $\gamma_\ap(1)=q_\ap/2h^*r_\ap$. Hence the coefficient of 
$k$ in Eq.~\eqref{eq:delta1(m-k)} equals $(q_\ap-2q_\ap(1))/2h^*r_\ap$. This settles {\sf (i)--(iii)}.
\\ 
By Proposition~\ref{prop:simmetri-d_ap}, $q_\ap(1)=q_\ap(d_\ap-1)$ for $d_\ap\ge 2$. Hence $q_\ap>2q_\ap(1)$ whenever $d_\ap\ge 3$.
\end{proof}
\begin{ex}    \label{ex:starshe-korn}
1)  If $d_\ap=2$, then all three possibilities may occur, cf. the good cases $(\GR{D}{n}, \ap_i)$ for 
$2\le i\le n-2$ and sufficiently large $n$. (Use data from Table~\ref{table:d=2c}.)
\\   \indent
2) Suppose that $\theta$ is fundamental, i.e., $\theta=\vp_{\widehat\ap}$.  
Then $\widehat\ap\in\Pi_l$, $d_{\widehat\ap}=2$, and $\Delta_{\widehat\ap}(2)=\{\theta\}$. Since 
$\g_{\widehat\ap}({\ge }1)$ is a Heisenberg Lie algebra, see~\cite[Sect.\,2]{jos76}, $\g_{\widehat\ap}(1)$ has an abelian subspace of half-dimension 
and the above computation applies. Here $m=2h^*-4$, $q_{\widehat\ap}=h^*-1$, 
and $q_{\widehat\ap}(2)=1$. Since $\theta$ is fundamental, 
$h^*\ge 4$ and hence $q_{\widehat\ap} < 2q_{\widehat\ap}(1)$. Then $\gamma_{\widehat\ap}(1)=
(h^*-1)/2h^*$ and $\chi=1/(4h^*)$. Thus, for $k\le m/2=h^*-2$, we obtain
$\delta_{\widehat\ap}(k)=\displaystyle k\frac{(h^*-1)}{2h^*}$ \ and 
\[
   \delta_{\widehat\ap}(m-k)=k\frac{(h^*-1)}{2h^*}+\frac{m(m-2k)}{4h^*}=
   \frac{(h^*-2)^2}{h^*} - k{\cdot}\frac{h^*-3}{2h^*} .
\]
3) For $\GR{C}{n}$ and $n\ge 2$, we have $\theta=2\vp_1$ and $\widehat\ap=\ap_1$ is short. Here 
one computes  that $2q_{\ap_1}(1)>q_{\ap_1}$ for $n>2$ and $\chi=1/(4h^*)$.
\end{ex}
Another application of our theory, especially of Theorem~\ref{thm:main3}, is the following result.

\begin{thm}    \label{thm:abel-complem}
If\/ $\ah\subset\g_\ap(1)$ is an abelian subspace and $\dim\ah=(1/2)\dim\g_\ap(1)$, then there is an
abelian subspace $\tilde\ah$ such that $\ah\oplus\tilde\ah=\g_\ap(1)$.
\end{thm}
\begin{proof}
As above, $m=\dim\g_\ap(1)$ and  $\be_\ap(0)=\be\cap\g_\ap(0)$.
\\ \indent
1.  Assume  that $\ah$ is a $\be_\ap(0)$-{\it stable\/} abelian subspace. In particular, $\ah$ is 
$\te$-stable. Therefore, there is a unique $\te$-stable complement $\tilde\ah$ to $\ah$ in $\g_\ap(1)$.
Then $\tilde\ah$ is $\be_\ap(0)^-$-stable.
Choose nonzero (poly)vectors $y\in\bigwedge^{m/2}\ah$ and $\tilde y\in\bigwedge^{m/2}\tilde\ah$.
Then $y$ (resp. $\tilde y$) is a highest (resp. lowest) weight vector in the $\g_\ap(0)$-module
$\bigwedge^{m/2}\g_\ap(1)$.
If $\mathsf{wt}(\cdot)$ stands for the $\te$-weight of a (poly)vector, then
\[
          \mathsf{wt}(y) + \mathsf{wt}(\tilde y)=|\Delta_\ap(1)|=q_\ap(1)\vp_\ap .
\]
As was already computed in Proposition~\ref{prop:delta-top},
\[
    \delta_\ap(m)=(q_\ap(1)\vp_\ap, q_\ap(1)\vp_\ap+2\rho_\ap(0))=q_\ap(1)^2(\vp_\ap,\vp_\ap)=\frac{mq_\ap(1)}{2h^*r_\ap} .
\]
Since $y$ is a highest weight vector in $\bigwedge^{m/2}\g_\ap(1)$ and $\ah$ is an abelian subspace,
\[
  \eus C_\ap(0)(y)=(\mathsf{wt}(y),\mathsf{wt}(y)+2\rho_\ap(0)){\cdot}y=\frac{m}{2}\gamma_\ap(1){\cdot}y.
\]
On the other hand, $\mathsf{wt}(\tilde y)$ is anti-dominant w.r.t. $\be_\ap(0)$. Hence the weight
$w_{\ap,0}(\mathsf{wt}(\tilde y))$ is already dominant and the $\eus C_\ap(0)$-eigenvalue of 
$\tilde y$ equals
\begin{multline*}
  \bigl(w_{\ap,0}(\mathsf{wt}(\tilde y)),w_{\ap,0}(\mathsf{wt}(\tilde y))+2\rho_\ap(0)\bigr)=
  \bigl(\mathsf{wt}(\tilde y),\mathsf{wt}(\tilde y)-2\rho_\ap(0)\bigr) \\
  =(q_\ap(1)\vp_\ap-\mathsf{wt}(y), q_\ap(1)\vp_\ap-\mathsf{wt}(y)-2\rho_\ap(0)) \\
  =q_\ap(1)^2(\vp_\ap,\vp_\ap)-2\bigl(q_\ap(1)\vp_\ap, \mathsf{wt}(y)+\rho_\ap(0)\bigr)+\bigl(\mathsf{wt}(y),\mathsf{wt}(y)+2\rho_\ap(0)\bigr)\\
  = \frac{mq_\ap(1)}{2h^*r_\ap}-2\bigl(q_\ap(1)\vp_\ap, \mathsf{wt}(y)\bigr)+\frac{m}{2}\gamma_\ap(1)\\
  =\frac{mq_\ap(1)}{2h^*r_\ap}-2q_\ap(1){\cdot}\frac{m}{2}{\cdot}(\vp_\ap,\ap)+\frac{m}{2}\gamma_\ap(1)
  = \frac{m}{2}\gamma_\ap(1) .
\end{multline*}
Here we used the facts that $(\vp_\ap,\rho_\ap(0))=0$ and 
$(\vp_\ap,\gamma)=(\vp_\ap,\ap)=(\ap,\ap)/2$ for any
$\gamma\in\Delta_\ap(1)$. By Theorem~\ref{thm:main3}, the equality
$\eus C_\ap(0)(\tilde y)=\frac{m}{2}\gamma_\ap(1){\cdot}\tilde y$ for an $m/2$-vector $\tilde y$
means that the $m/2$-dimensional subspace $\tilde \ah$ is abelian.
\\ \indent 
2. If $\ah$ is not $\be_\ap(0)$-stable, then we consider the $B_\ap(0)$-orbit of $\{\ah\}$ in the
Grassmannian of $m/2$-dimensional subspaces of $\g_\ap(1)$. By the Borel fixed-point theorem, the
closure of this orbit contains a $B_\ap(0)$-fixed point, i.e., a $\be_\ap(0)$-stable (abelian) subspace,
say $\ah_1$. If $\tilde\ah_1$ is the complementary abelian subspace for $\ah_1$, as in part~1, then,
by continuity,  it is also a complementary subspace for some element of the orbit 
$B_\ap(0){\cdot}\{\ah\}$.
\end{proof}

Previous results show that it is helpful to know whether $\g_\ap(1)$ has an abelian 
subspace of half-dimension, if $d_\ap{>}1$. We say that $\ap\in\Pi$ is {\it good} if this is the case; 
otherwise, $\ap$ is {\it bad}. In many cases, a $(\BZ,\ap)$-grading is also the Dynkin grading associated 
with a {\it strictly odd\/} nilpotent element of $\g$, see~\cite[Sect.\,1]{e-j-k}. Then the relevant good cases 
have been determined in~\cite{e-j-k}. However, some work is still needed for the $(\BZ,\ap)$-gradings that are not Dynkin. For instance, if $\g$ is exceptional, then one has to handle the 
possibilities $(\GR{E}{7}, \ap_3 \text{ or } \ap_7)$  and $(\GR{E}{6}, \ap_2 \text{ or } \ap_4)$. Combining
our computations with \cite{e-j-k}, we describe below the bad cases for all $\g$. For each bad case, the maximal dimension of an abelian subspace, $\dim\ah_{\sf max}$, is given.
Note that in order to compute $\dim\ah_{\sf max}$, it suffices to consider only $\be_\ap(0)$-stable
abelian subspaces of $\g_\ap(1)$, cf. part 2) in the proof of Theorem~\ref{thm:abel-complem}.

\textbullet\  For the classical series, we have $d_\ap\le 2$. If $\g=\spn$ or $\sone$, then 
all $\ap\in\Pi$ with $d_\ap=2$ are good. If $\g=\sono$, $n\ge 3$, then the bad cases occur for $\ap_i$ with $3\le i\le n$. Here $\dim\g_{\ap_i}(1)=2i(n-i)+i$ 
and $\dim\ah_{\sf max}=i(n-i)+1$. Note also that, for $\sone$ and $\sono$, the $(\BZ,\ap_i)$-grading is 
Dynkin and associated with a strictly odd nilpotent if and only if $i$ is even (and $d_{\ap_i}=2$).

\textbullet\  For the exceptional algebras, we gather the bad cases in Table~\ref{table:bad-cases},
where we write $m_\ap$ for $\dim\g_\ap(1)$.
\begin{table}[ht]
\caption{Exceptional Lie algebras, the bad cases}   \label{table:bad-cases}
\begin{center}
\begin{tabular}{>{$}c<{$} >{$}c<{$} >{$}c<{$} >{$}c<{$} >{$}c<{$} c ||>{$}c<{$} >{$}c<{$} >{$}c<{$} >{$}c<{$}  >{$}c<{$} c | }
\g & \ap & d_\ap & m_\ap & \dim\ah_{\sf max} & \cite{e-j-k} & \g & \ap & d_\ap & m_\ap & \dim\ah_{\sf max} 
& \cite{e-j-k}\\ \hline \hline
\GR{E}{7} & \ap_3 & 3 & 30 &12 & - 
    & \GR{E}{8} & \ap_3 & 4 & 48 &16 & +   \\
      & \ap_7 & 2 & 35 &15 & - 
     &  & \ap_4 & 5 & 40 & 16 & +      \\  \cline{1-6} 
   \GR{F}{4} & \ap_1 & 2 & 8 & 2 & + 
      &  & \ap_7 & 2 & 64 &22 & +    \\   
       & \ap_2 & 4 & 6 & 2 & + 
       &  & \ap_8 & 3 & 56 & 21 & +\\   \hline
\end{tabular}
\end{center}
\end{table}

The data in Table~\ref{table:bad-cases} also mean that the non-Dynkin cases 
$(\GR{E}{6}, \ap_2 \text{ or } \ap_4)$ are good, cf. also Example~\ref{F4-ap1}(2). 
The signs $+/-$ indicate whether that item represents a Dynkin grading (=\,is considered 
in~\cite{e-j-k}).

Our methods for constructing abelian subspaces of $\g(1)$, partly  for arbitrary $\BZ$-gradings, are 
described below. This provides another approach to some of calculations in~\cite{e-j-k} and also
natural descriptions of abelian subspaces of maximal dimension.

\begin{lm}    \label{lm:ab-for-d=2}
Let $\g=\bigoplus_{i=-2}^2\g(i)$ be a $\BZ$-grading of height~2 and $\be(0)$ a Borel 
subalgebra of\/ $\g(0)$. If\/ $\ah$ is a $\be(0)$-stable abelian subspace of\/ $\g(1)$, then 
$\ah\oplus\g(2)$ is an abelian $\be$-ideal of\/ $\g$, where $\be=\be(0)\oplus\g(1)\oplus\g(2)$.
In particular, if\/ $\dim\ah$ is maximal, then $\ah\oplus\g(2)$ is a maximal abelian $\be$-ideal.
\end{lm}

Since the maximal abelian $\be$-ideals are known \cite[Sect.\,4]{adv01}, one readily obtains an upper 
bound on $\dim\ah$. Actually, this allows us to determine $\dim\ah_{\sf max}$ for all $\BZ$-gradings of 
height~$2$. The next observation applies to $(\BZ,\ap)$-gradings of any height. 

\begin{prop}     \label{prop:d_ap-&-d_beta}
Given $\ap\in\Pi$, suppose that $\g_\ap(2)\cap \g_\beta(2)=\{0\}$ for some $\beta\in\Pi$. Then 
$\ah_{\ap,\beta}:=\g_\ap(1)\cap\g_\beta(1)$ is abelian. Moreover, if also $\g_\ap(1)\cap \g_\beta(2)=\{0\}$, 
then $\ah_{\ap,\beta}$ is a $\be_\ap(0)$-stable abelian subspace of $\g_\ap(1)$. 
\end{prop}
The proof  is straightforward and left to the reader. 

\noindent
There are interesting instances of such phenomenon and we provide below some illustrations 
to our method. It turns out {\sl a posteriori} that the two assumptions of the above proposition imply
that $d_\beta<d_\ap$. However, even if Proposition~\ref{prop:d_ap-&-d_beta} applies, then the
abelian subspace $\ah_{\ap,\beta}$ does not necessarily have the maximal dimension.

\begin{rmk}      
\label{ex:instances}
{\sf (i)} \ Note that $\g_\ap(1)\cap\g_\beta(1)\ne\{0\}$ for {\bf all} pairs 
$\{\ap,\beta\}\subset\Pi$. For, take the unique chain in the Dynkin diagram joining $\ap$ 
and $\beta$. The sum of simple roots in this chain is a root, denoted by $\mu_{\ap,\beta}$, and 
it is clear that $\mu_{\ap,\beta}\in \Delta_\ap(1)\cap\Delta_\beta(1)$. Clearly,  
$\h=\g_\ap(0)\cap\g_\beta(0)$ is a Levi subalgebra in $\p_\ap\cap \p_\beta$ and the set of simple roots of $\h$ is $\Pi\setminus\{\ap,\beta\}$.
By~\cite[Theorem\,0.1]{ko10} (cf. Section~\ref{subs:Z-grad-versus}), 
$\g_\ap(i)\cap\g_\beta(j)$ is a simple $\h$-module
for any $(i,j)$. Obviously, $\mu_{\ap,\beta}$ is the lowest weight of the $\h$-module $\ah_{\ap,\beta}$, so it is an easy task to compute $\dim\ah_{\ap,\beta}$ for any pair $\{\ap,\beta\}\subset\Pi$.
\\ \indent
{\sf (ii)} \ if $d_\beta=1$, then $\g_\beta(1)$ is a maximal abelian $\be$-ideal and the assumptions of
Proposition~\ref{prop:d_ap-&-d_beta} are satisfied. 
\\ \indent
{\sf (iii)} \ Another possibility for applying Proposition~\ref{prop:d_ap-&-d_beta} is that in which 
$d_\ap\ge 3$ (hence $\g$ is exceptional) and 
$\beta=\widehat\ap$ is the unique simple root such that $(\theta,\widehat\ap)\ne 0$. Then 
$\Delta_{\widehat\ap}(2)=\{\theta\}$, while $\hot_\ap(\theta)\ge 3$. Hence
$\Delta_{\widehat\ap}(2)\cap(\Delta_\ap(1)\cup\Delta_\ap(2))=\varnothing$.
\end{rmk}

\begin{ex}    \label{ex:illustr}
1) Let $\theta$ be a multiple of a fundamental weight (i.e., $\Delta$ is not of type $\GR{A}{n}$, 
$n\ge 2$) and, as usual, $(\theta,\widehat\ap)\ne 0$. For the
$(\BZ,\widehat\ap)$-grading, one has $d_{\widehat\ap}=2$,
$\g_{\widehat\ap}(2)=\g^\theta$, and $\g_{\widehat\ap}({\ge}1)$ is a Heisenberg Lie algebra. Here 
$\dim\g_{\widehat\ap}(1)=2h^*-4$ and it follows from~\cite[Sect.\,3]{jems} that, for any maximal
abelian $\be$-ideal $\eus I$, we have $\dim\bigl(\eus I\cap\g_{\widehat\ap}({\ge}1)\bigr)=h^*-1$. Hence
$\dim\bigl(\eus I\cap\g_{\widehat\ap}(1)\bigr)=h^*-2=(1/2)\dim\g_{\widehat\ap}(1)$. Thus, 
$\g_{\widehat\ap}(1)\cap\eus I$ is an abelian $\be_{\widehat\ap}(0)$-stable subspace of $\g_\ap(1)$
of half-dimension for {\bf any} maximal abelian ideal $\eus I$. Actually, different $\eus I$'s yield
different subspaces $\g_{\widehat\ap}(1)\cap\eus I$.
\\ \indent
2) If $\g$ is exceptional, then $\widehat\ap$ is an {\bf extreme} root in the Dynkin diagram. Let $\ap\in\Pi$ be the unique root adjacent to $\widehat\ap$. Then 
$1=(\theta,{\widehat\ap}^\vee)=d_{\widehat\ap}(\widehat\ap,{\widehat\ap}^\vee)+d_\ap(\ap,{\widehat\ap}^\vee)=4-d_\ap$. 
Hence $d_\ap=3$  and therefore $\ah_{\ap,\widehat\ap}$ is a $\be_\ap(0)$-stable abelian subspace of 
$\g_\ap(1)$, cf. Proposition~\ref{prop:d_ap-&-d_beta} and Remark~\ref{ex:instances}(iii). We claim that 
$(\g,\ap)$ is a good case.
For, in this case, $\g_\ap(0)'=\tri\dotplus\q$, where $\tri$ corresponds to $\widehat\ap$ and the simple
roots of the semisimple algebra $\q$ are $\Pi\setminus\{\ap,\widehat\ap\}$. Here $\g_\ap(1)\simeq \bbk^2\otimes V$ as 
$\g_\ap(0)'$-module, where 
$\bbk^2$ is the standard $\tri$-module and $V$ is a $\q$-module. Therefore, if 
$\gamma\in\Delta_\ap(1)$, then $\hot_{\widehat\ap}(\gamma)\in\{0,1\}$; and if
$\hot_{\widehat\ap}(\gamma)=0$, then $\gamma+\widehat\ap\in \Delta_\ap(1)$, and vice versa.
It follows that  $\Delta_\ap(1)\cap \Delta_{\widehat\ap}(1)=
\{\gamma\in\Delta_\ap(1)\mid \hot_{\widehat\ap}(\gamma)=1\}$ contains exactly half of the roots in 
$\Delta_\ap(1)$. Thus, $\dim\ah_{\ap,\widehat\ap}=\frac{1}{2}\dim\g_\ap(1)$.
\\ \indent
3) For $\GR{E}{n}$, one verifies that if $\beta\in\Pi$ is {\bf any} extreme root of the Dynkin diagram and 
$\ap$ is the unique root adjacent to $\beta$, then $\ah_{\ap,\beta}$ is abelian and 
$\dim\ah_{\ap,\beta}=\frac{1}{2}\dim\g_\ap(1)$. The last equality is again explained by the fact that here  
$\g_{\ap}(0)'\simeq\tri\dotplus\q$ and $\g_\ap(1)\simeq \bbk^2\otimes V$.
\end{ex}

\begin{ex}   \label{F4-ap1}
1) For $(\GR{F}{4},\ap_1)$, we have $d_{\ap_1}=2$, $\dim\g_{\ap_1}(1)=8$, and $\dim\g_{\ap_1}(2)=7$. 
If $\eus I$ is an abelian $\be$-ideal, then $\dim \eus I\le 9$. Hence $\dim\ah\le 9-7=2$. Actually, 
$\dim(\eus I\cap\g_{\ap_1}(1))=2$, if $\dim\eus I=9$.
\\ \indent
2) For $(\GR{E}{6},\ap_2)$, we have $d_{\ap_2}=2$ and $\dim\g_{\ap_2}(1)=20$. Here $d_{\ap_1}=1$ 
and hence $\g_{\ap_1}(1)$ is a (maximal) abelian $\be$-ideal. Since 
$\dim (\g_{\ap_1}(1)\cap\g_{\ap_2}(1))=10$, this is a good case.
\\ \indent
3) For $(\GR{E}{7},\ap_7)$, we have $d_{\ap_7}=2$, $\dim\g_{\ap_7}(1)=35$, and $\dim\g_{\ap_7}(2)=7$. 
Here $d_{\ap_1}=1$ and $\g_{\ap_1}(1)$ is the maximal abelian ideal of maximal dimension $27$. In this 
case, $\Pi\setminus \{\ap_1,\ap_7\}$ is the Dynkin diagram of type $\GR{A}{5}$ and
$\g_{\ap_7}(1)\cap\g_{\ap_1}(1)$ is the simple $SL_6$-module $\bigwedge^2(\bbk^6)$, of dimension 15.
The minimal (resp. maximal) root in $\Delta_{\ap_7}(1)\cap\Delta_{\ap_1}(1)$ is
\raisebox{-1.7ex}{\begin{tikzpicture}[scale= .7, transform shape]
\node (a) at (0,0) {\bf 1};
\node (b) at (.3,0) {\bf 1};
\node (c) at (.6,0) {\bf 1};
\node (d) at (.9,0) {\bf 1};
\node (e) at (1.2,0) {\bf 0};
\node (f) at (1.5,0) {\bf 0};
\node (g) at (.9,-.5) {\bf 1};
\end{tikzpicture}} (resp. 
\raisebox{-1.7ex}{\begin{tikzpicture}[scale= .7, transform shape]
\node (a) at (0,0) {\bf 1};
\node (b) at (.3,0) {\bf 2};
\node (c) at (.6,0) {\bf 3};
\node (d) at (.9,0) {\bf 3};
\node (e) at (1.2,0) {\bf 2};
\node (f) at (1.5,0) {\bf 1};
\node (g) at (.9,-.5) {\bf 1};
\end{tikzpicture}}\!\!). 
For the other maximal abelian ideals $\eus I$, one obtains 
$\dim(\g_{\ap_1}(1)\cap \eus I)\le 15$. 
\end{ex}

\begin{rmk}   \label{rem:vsegda-if-d>3}
If $\g$ is exceptional and $d_\ap\ge 3$, then one can always find $\beta\in\Pi$ such that $d_\beta<d_\ap$,  Proposition~\ref{prop:d_ap-&-d_beta} applies, and $\ah_{\ap,\beta}$ has the required dimension,
i.e., $(1/2)\dim\g_\ap(1)$ in the good cases and the numbers $\dim\ah_{\sf max}$ from 
Table~\ref{table:bad-cases} in the bad cases. 
For instance, one takes 

\textbullet\ 
for $\GR{E}{6}$: \ $\beta=\ap_6$ if $\ap=\ap_3$;

\textbullet\ 
for $\GR{E}{7}$: \ $\beta=\ap_6$ if $\ap=\ap_3$ or $\ap_5$;  
$\beta=\ap_7$ if $\ap=\ap_4$; 

\textbullet\ 
for $\GR{E}{8}$: \ $\beta=\ap_1$ if $\ap=\ap_2$ or $\ap_3$ or $\ap_8$; 
$\beta=\ap_7$ if $\ap=\ap_4$ or $\ap_6$; 
$\beta=\ap_8$ if $\ap=\ap_5$;

\textbullet\ 
for $\GR{F}{4}$: \ $\beta=\ap_4$ if $\ap=\ap_2$ or $\ap_3$.
\end{rmk}

\section{Variations on themes of the "strange formula"}
\label{sect:FdV}

\noindent
Let $\g$ be a reductive algebraic Lie algebra.
For any orthogonal $\g$-module $\eus V$, there is another $\g$-module, denoted
$\spin(\eus V)$. Roughly speaking, one takes the spinor representation of $\sov$ and restricts it to 
$\g\subset\sov$.
It has the property that 
\[  \textstyle
\bigwedge^\bullet \eus V\simeq \begin{cases}
\spin(\eus V)\otimes\spin(\eus V), & \text{ if $\dim\eus V$ is even} \\
2(\spin(\eus V)\otimes\spin(\eus V)), & \text{ if $\dim\eus V$ is odd}   \end{cases} \ ,
\] 
see~\cite[Section\,2]{tg01}.  Moreover, extracting further a numerical  factor from the $\g$-module 
$\spin(\eus V)$, one 
can uniformly write $\spin(\eus V)=2^{[m(0)/2]}\spin_0(\eus V)$ and then  \\[.4ex]
\centerline{
$\bigwedge^\bullet \eus V\simeq 2^{m(0)}{\cdot}\bigl(\spin_0(\eus V)\otimes\spin_0(\eus V)\bigr)$,}
\\[.8ex]
where $m(0)$ is the multiplicity of the zero weight in $\eus V$. There are only few orthogonal simple 
$\g$-modules
$\eus V$ such that $\spin_0(\eus V)$ is again simple, see~\cite[Section\,3]{tg01}. A notable example is 
that $\spin_0(\g)=\eus V_\rho$ for any simple Lie algebra $\g$, cf. Introduction.

From now on, $\g$ is again a simple Lie algebra. Let $\g=\g_0\oplus\g_1$ be a $\BZ_2$-grading. 
The corresponding involution of $\g$ is denoted by $\sigma$. Write 
$\eus C_0\in \eus U(\g_0)$ for the Casimir element associated with $\Phi\vert_{\g_0}$.
Then the $\eus C_0$-eigenvalue on $\g_1$ equals $1/2$, see~\cite{jlms01} and Proposition~\ref{prop:lms01}.

The $\g_0$-module $\g_1$ is orthogonal, and we are interested now in the $\spin$-construction for 
$\eus V=\g_1$. Then $m(0)=\rk\g-\rk\g_0$, and $m(0)=0$ if and only if $\sigma$ is an inner involution. 
Hence $\spin_0(\g_1)=\spin(\g_1)$ whenever $\sigma$ is inner. There is an explicit description of the 
irreducible constituents of $\spin_0(\g_1)$ in~\cite[Sections 5,\,6]{tg01}. This also implies that 
$\spin_0(\g_1)$ is always reducible if $\sigma$ is inner.
Although $\spin_0(\g_1)$ can be highly reducible, it is proved in~\cite[Theorem\,7.7]{tg01} that
$\eus C_0$ acts scalarly on $\spin_0(\g_1)$ and the corresponding eigenvalue is 
\[
    \gamma_{\spin_0(\g_1)}=(\rho,\rho)-(\rho_0,\rho_0), 
\]
where $\rho_0$ is the half-sum of positive roots of 
$\g_0$. Of course, we adjust here the Cartan subalgebras, $\te_0\subset \g_0$  and $\te\subset \g$ such that $\te_0\subset\te$. Then we can assume that $\te^*_0\subset \te^*$, etc.
In this section, we show that the $\eus C_0$-eigenvalue $\gamma_{\spin_0(\g_1)}$ has another nice uniform expression and that this is related to the ``strange formula of Freudenthal--de Vries''  (={\sf\bfseries sfFdV}).

\begin{thm}    \label{thm:FdV-inner}
Let $\sigma$ be an inner involution of $\g$ and $\g=\g_0\oplus\g_1$  the corresponding
$\BZ_2$-grading. Then
\[
    \gamma_{\spin_0(\g_1)}=(\dim\g_1)/16 .
\] 
\end{thm}
\begin{proof} Our argument relies on the theory developed in Section~\ref{sect:5/2}
and a relationship between involutions (=\, $\BZ_2$-gradings) and
$(\BZ{,}\ap)$-gradings of height at most $2$.
\\ \indent
Suppose that $d_\ap=\hot_\ap(\theta)\le 2$ and let
$\g=\bigoplus_{i=-d}^d\g_\ap(i)$ be the corresponding $\BZ$-grading. Letting $\g_0=\g_\ap(-2)\oplus \g_\ap(0)\oplus\g_\ap(2)$
and $\g_1=\g_\ap(-1)\oplus\g_\ap(1)$, we obtain a $\BZ_2$-grading (obvious). Since $\rk\g=\rk\g_0$, 
this involution is inner. The point is that {\bf all inner} involutions of $\g$ are obtained in this way, as 
follows from Kac's description in~\cite{kac}, cf. also~\cite[Ch.\,3,\ \S3.7]{t41}. 
Different simple 
roots $\ap,\beta$ with $d_\ap=d_\beta=2$ lead to ``one and the same'' $\BZ_2$-grading if and only if there is an automorphism of the extended Dynkin diagram of $\g$ that takes 
$\ap$ to $\beta$. 
If $d_\ap=1$, then $\g_0=\g_\ap(0)$ is not semisimple, whereas $\g_0$ is semisimple for $d_\ap=2$. 
The  subalgebras $\g_0$ corresponding to $\ap$ with $d_\ap=2$ are indicated in Tables~\ref{table:d=2c}
and \ref{table:d=2e}.

We can express $\rho_0$ and $\rho_1=\rho-\rho_0$ via the data related to the $\BZ$-grading.
That is, $\rho_0=\frac{1}{2}\bigl(|\Delta_\ap^+(0)|+|\Delta_\ap(2)|\bigr)$ and 
$\rho_1=\frac{1}{2}|\Delta_\ap(1)|=\frac{1}{2}q_\ap(1)\vp_\ap$. Since $(\vp_\ap,\mu)=0$ for any
$\mu\in \Delta^+_\ap(0)$, we have $(\rho_1,|\Delta_\ap^+(0)|)=0$ and therefore
\[
    (\rho,\rho)-(\rho_0,\rho_0)=(\rho_1,2\rho_0+\rho_1)=
    (\rho_1, |\Delta_\ap(2)|+\rho_1)=
    \frac{1}{4}\bigl(q_\ap(1)\vp_\ap, (q_\ap(1)+2q_\ap(2))\vp_\ap\bigr).
\]
Now, $q_\ap(1)+2q_\ap(2)=q_\ap+q_\ap(2)=h^*r_\ap$, see Corollary~\ref{cor:sravnenie-2} with $d_\ap=2$.
For $d_\ap=1$, one has $r_\ap=1$ and again $q_\ap(1)+2q_\ap(2)=q_\ap=h^*$, see Theorem~\ref{thm:otmetki}(2$^o$). So, if $d_\ap\le 2$, then
\begin{multline*}
  \gamma_{\spin_0(\g_1)}=\frac{h^*r_\ap}{4}(|\Delta_\ap(1)|, \vp_\ap)=
  \frac{h^*r_\ap}{4}\sum_{\mu\in\Delta_\ap(1)}(\mu,\vp_\ap)=\frac{h^*r_\ap}{4}\sum_{\mu\in\Delta_\ap(1)}
  (\ap,\vp_\ap) \\
  =\frac{h^*r_\ap}{4}\dim\g_\ap(1)\frac{(\ap,\ap)}{2}=\frac{1}{8}\dim\g_\ap(1)=\frac{1}{16}\dim\g_1 .
\end{multline*}
Here we use the fact that $\hot_\ap(\mu)=1$ for any $\mu\in \Delta_\ap(1)$ and hence
$(\mu,\vp_\ap)=(\ap,\vp_\ap)$.
\end{proof}

Actually, the previous result holds for any involution of $\g$, see below. This can be regarded as both an 
application and generalisation of the {\sf\bfseries sfFdV}. 
However, whereas the proof of Theorem~\ref{thm:FdV-inner} does not refer to the {\sf\bfseries sfFdV}, the 
general argument below, which applies to arbitrary involutions, explicitly relies on the {\sf\bfseries sfFdV}. 

\begin{thm}      \label{thm:general-g_1}
For any involution of a simple Lie algebra $\g$, we have $\gamma_{\spin_0(\g_1)}=(\dim\g_1)/{16}$.
\end{thm}
\begin{proof}
Write $\g_0=(\bigoplus_i\h_i)\oplus \ce$ as the sum of simple ideals $\{\h_i\}$ and possible centre $\ce$.
To prove the assertion, we need basically the following three facts on $\eus C_0=\eus C_{\g_0}$:
\begin{itemize}
\item $\tr_\g(\eus C_0)=\dim\g_0$, see~\cite{jlms01} and  Proposition~\ref{prop:lms01}{\sf (i)};
\item  the $\eus C_0$-eigenvalue on $\g_1$ equals $1/2$, see  Proposition~\ref{prop:lms01}{\sf (iv)};
\item  the usual {\sf\bfseries sfFdV} for $\g$ and for the simple ideals of $\g_0$.
\end{itemize}
Let $\eus C_{\h_i}$ be the "canonical" Casimir element for $\h_i$. Then $\eus C_{\h_i}$ has the eigenvalue $1$ on $\h_i$. Since $\h_i$ is an ideal of $\g_0$,
there is a transition factor, $T_i$, between the eigenvalues of
$\eus C_{\h_i}$ and $\eus C_0$ on the $\h_i$-modules, cf. the proof of Theorem~\ref{thm:gen-formula}.
Actually, we even know that $T_i=\frac{h^*(\h_i)}{h^*{\cdot}\ind(\h_i\hookrightarrow\g)}$, but this precise 
value is of no importance in the rest of the argument.
Since the $\eus C_0$-eigenvalue on $\h_i$ is $T_i$, we have
$\tr_{\h_i}{\eus C_0}=T_i {\cdot}\dim\h_i$.  Hence
\[
   \dim\g_0= \tr_\g(\eus C_0)=\tr_{\g_1}(\eus C_0)+\tr_{\g_0}(\eus C_0)=
   \frac{1}{2}\dim\g_1+\textstyle \sum_i T_i {\cdot}\dim\h_i .
\]
On the other hand, let $\rho_i$ be the half-sum of the positive roots of $\h_i$.
Then $\rho_0=\sum_i\rho_i$, $(\rho_i,\rho_j)=0$ for $i\ne j$, and  
$(\rho_i,\rho_i)=T_i{\cdot} (\dim\h_i/24)$ in view of the {\sf\bfseries sfFdV} for $\h_i$.
Thus,
\begin{multline*}
 \gamma_{\spin_0(\g_1)}=(\rho,\rho)-(\rho_0,\rho_0)=
 \frac{1}{24}\dim\g- \frac{1}{24}(\textstyle \sum_i T_i\dim\h_i)   \displaystyle
  \\  = \frac{1}{24}\dim\g-  \frac{1}{24} (\dim\g_0-\frac{1}{2}\dim\g_1)=\frac{1}{16}\dim\g_1 .
  \qedhere
\end{multline*}
\end{proof}

\begin{rmk}      \label{rem:g+g}
The adjoint representation of $\g$ can be regarded as the isotropy representation related to the
permutation, $\tau$, of the summands in $\tilde\g=\g\dotplus\g$. Here $\tilde\g$ is not simple, 
but $\tilde\g_0\simeq \g$ is. In this situation, there is an analogue of Theorem~\ref{thm:general-g_1} for
$(\tilde\g, \tau)$, and we demonstrate below that it is equivalent
to the {\sf\bfseries sfFdV} for $\g$.  In other words, under proper adjustments of bilinear forms and
$\eus C_{\tilde\g_0}$,  
the formula of Theorem~\ref{thm:general-g_1} for $(\tilde\g,\tau)$ transforms  into 
the {\sf\bfseries sfFdV} for $\g$, and vice versa. 
One of the main reasons is that, for the orthogonal
$\g$-module $\g$, one has $\spin_0(\g)=\eus V_\rho$, see~\cite[(5.9)]{ko61} and~\cite[(2.5)]{tg01}. 
\\ \indent
Recall that $\g=\ut\oplus\te\oplus\ut^-$ and $\be=\ut\oplus\te$. Then
$\tilde\be=\be\dotplus\be$ and $\tilde\te=\te\dotplus\te$.
In what follows, various objects related to the two factors of $\tilde\g$ will be marked with the  
superscripts `$(1)$' and `$(2)$'. As above,  $\Phi$ is the Killing form on $\g$ and $(\ ,\ )$ is the induced 
(canonical) bilinear form on $\te^*$.
Let $\tilde\Phi=\Phi^{(1)}\dotplus\Phi^{(2)}$ be the invariant bilinear form on $\tilde\g$. The induced
bilinear form on $\tilde\te^*$ is denoted by $(\ ,\ )^{\sim}$.  Then the Casimir element 
$\eus C_{\tilde\g_0}$ is defined via the restriction of $\tilde\Phi$ to $\tilde\g_0$. We have 
$\tilde\rho=\rho^{(1)}+\rho^{(2)}$ and these two summands are orthogonal w.r.t. $(\ ,\ )^{\sim}$; \ 
hence 
$(\tilde\rho ,\tilde\rho )^{\sim}=(\rho^{(1)} ,\rho^{(1)})^{\sim}+
(\rho^{(2)} ,\rho^{(2)})^{\sim}=2(\rho,\rho)$. 
The Cartan subalgebra $\tilde\te_0$
of $\tilde\g_0$ is diagonally imbedded in $\tilde\te$, hence so are the roots of $\tilde\g_0$ in 
$\tilde\te_0^*$. In particular, both components of $\tilde\rho_0$ are equal to $\rho$. 
Identifying ${\tilde\g_0}$ with $\g$ via the projection to the first component, one readily obtains that
$\tilde\Phi\vert_{\tilde\g_0}=2\Phi$. It then follows from Lemma~\ref{lm:dve-formy} that one obtains 
the relation with the {\bf inverse} coefficient for the corresponding canonical bilinear forms.
This yields our key equality
\[
    (\tilde\rho_0 ,\tilde\rho_0 )^{\sim}=\frac{1}{2}(\rho,\rho) .
\]
Afterwards, using the {\sf\bfseries sfFdV} for $\g$, we obtain
\[
    (\tilde\rho ,\tilde\rho )^{\sim}-(\tilde\rho_0 ,\tilde\rho_0 )^{\sim}=2(\rho,\rho)-\frac{1}{2}(\rho,\rho)=
    \frac{3}{2}(\rho,\rho)=\frac{1}{16}\dim\g=\frac{1}{16}\dim\tilde\g_1 .
\]
And by \cite[Theorem\,7.7]{tg01}, the $\eus C_{\tilde\g_0}$-eigenvalue on $\spin_0(\tilde\g_1)$ equals
$(\tilde\rho ,\tilde\rho )^{\sim}-(\tilde\rho_0 ,\tilde\rho_0 )^{\sim}$. Thus, the equality of 
Theorem~\ref{thm:general-g_1} remains true for the involution $\tau$ of $\tilde\g$, and 
we have just shown that this equality is equivalent to the {\sf\bfseries sfFdV}.
\end{rmk}
This certainly means that it is of great interest to find a proof of Theorem~\ref{thm:general-g_1} and 
its analogue for the semisimple Lie algebra $\tilde\g=\g\dotplus\g$ that is independent of the
{\sf\bfseries sfFdV}.

\appendix
\section{The eigenvalues $\gamma_\ap(i)$ for $\ap\in\Pi$ with $d_\ap\ge 2$}   
\label{sect:App}

\noindent
Tables~\ref{table:d=2c}--\ref{table:d>5}  below 
provide the eigenvalues $\{\gamma_\ap(i)\}$ and integers $\{q_\ap(i)\}$ for the $(\BZ{,}\ap)$-gradings 
of all simple Lie algebras with $d_\ap\ge 2$. For, if $d_\ap=1$, then $\gamma_\ap(1)=1/2$ and
$q_\ap=q_\ap(1)=h^*$.
In the last column of Tables~\ref{table:d=2c}
and \ref{table:d=2e}, we point out the fixed-point (semisimple) subalgebra $\g_0$ for the
corresponding inner involution of $\g$,  cf. Section~\ref{sect:FdV}.

\begin{table}[ht]
\caption{Classical Lie algebras, $d_\ap=2$}   \label{table:d=2c}
\begin{center}
\begin{tabular}{>{$}c<{$} |>{$}c<{$} >{$}c<{$}| >{$}c<{$} >{$}c<{$}| >{$}c<{$} >{$}c<{$}  |>{$}c<{$} c }
\g,\ap & \gamma_\ap(1) & \gamma_\ap(2)  & q_\ap(1) & q_\ap(2) & h^* & r_\ap &  \g_0 & 
  \\ \hline\hline
\GR{B}{n},\ap_i  \rule{0pt}{3.8ex} & \displaystyle\frac{2n{-}i}{2(2n{-}1)} & \displaystyle\frac{i{-}1}{2n{-}1} & 2n{-}2i{+}1 & i{-}1 & 2n{-}1 & 1 &  \GR{D}{i}{\times}\GR{B}{n-i} & $2{\le} i{\le} n{-}1$ \\  
\GR{B}{n},\ap_n \rule{0pt}{3.8ex} & \displaystyle\frac{n}{2(2n{-}1)} & \displaystyle\frac{2n{-}2}{2(2n{-}1)} & 2 & 2n{-}2 & 2n{-}1 & 2 &  \GR{D}{n} & \\  
\GR{C}{n},\ap_i \rule{0pt}{3.8ex} & 
\displaystyle\frac{2n{+}1{-}i}{4(n{+}1)} & \displaystyle\frac{i{+}1}{2(n{+}1)} & 2n{-}2i & i{+}1 & n{+}1 & 2 &  \GR{C}{i}{\times}\GR{C}{n-i} & $1{\le} i{\le} n{-}1$ \\  
\GR{D}{n},\ap_i \rule{0pt}{3.8ex} & \displaystyle\frac{2n{-}1{-}i}{2(2n{-}2)} & 
\displaystyle\frac{i{-}1}{2n{-}2} & 2n{-}2i & i{-}1 & 2n{-}2 & 1 &  \GR{D}{i}{\times}\GR{D}{n-i}  
& $2{\le} i{\le} n{-}2$ \\  
\hline 
\end{tabular}
\end{center}
\end{table}

Recall that the numbering of simple roots follows \cite{t41}.

\begin{table}[ht]
\caption{Exceptional Lie algebras, $d_\ap=2$}   \label{table:d=2e}
\begin{center}
\begin{tabular}{>{$}c<{$} |>{$}c<{$} >{$}c<{$}| >{$}c<{$} >{$}c<{$} >{$}c<{$} >{$}c<{$}  |>{$}c<{$} }
\g,\ap & \gamma_\ap(1) & \gamma_\ap(2)  & q_\ap(1) & q_\ap(2) & h^* & r_\ap &  \g_0
  \\ \hline\hline
\GR{E}{6},\ap_2  \rule{0pt}{2.4ex} & 9/24 & 6/24 & 6 & 3 & 12 & 1 &  \GR{A}{5}\times \GR{A}{1}\\  
\GR{E}{6},\ap_6 \rule{0pt}{2.4ex} & 11/24 & 2/24 & 10 & 1 & 12 & 1 &  \GR{A}{5}\times \GR{A}{1}\\  
\GR{E}{7},\ap_2 \rule{0pt}{2.4ex} & 13/36 & 10/36 & 8 & 5 & 18 & 1 & \GR{D}{6}\times \GR{A}{1}\\  
\GR{E}{7},\ap_6 \rule{0pt}{2.4ex} & 17/36 & 2/36 & 16 & 1 & 18 & 1 & \GR{D}{6}\times \GR{A}{1} \\  
\GR{E}{7},\ap_7 \rule{0pt}{2.4ex} & 14/36 & 8/36 & 10 & 4 & 18 & 1 &  \GR{A}{7} \\ 
\GR{E}{8},\ap_1 \rule{0pt}{2.4ex} & 29/60 & 2/60 & 28 & 1 & 30 & 1 & \GR{E}{7}\times \GR{A}{1} \\ 
\GR{E}{8},\ap_7 \rule{0pt}{2.4ex} & 23/60 & 14/60 & 16 & 7 & 30 & 1 & \GR{D}{8} \\
\GR{F}{4},\ap_1 \rule{0pt}{2.4ex} & 11/36 & 14/36 & 4 & 7 & 9 & 2 & \GR{B}{4} \\ 
\GR{F}{4},\ap_4 \rule{0pt}{2.4ex} & 4/9 & 1/9 & 7 & 1 & 9 & 1 & \GR{C}{3}\times \GR{A}{1}   \\ 
\GR{G}{2},\ap_2 \rule{0pt}{2.4ex} & 3/8 & 2/8 & 2 & 1 & 4 & 1 & \GR{A}{1}\times \GR{A}{1}  \\  \hline 
\end{tabular}
\end{center}
\end{table}

\begin{table}[ht]
\caption{Exceptional Lie algebras, $d_\ap=3$}   \label{table:d=3}
\begin{center}
\begin{tabular}{>{$}c<{$}|>{$}c<{$}>{$}c<{$}>{$}c<{$}|>{$}c<{$}>{$}c<{$}>{$}c<{$}|>{$}c<{$}>{$}c<{$}}
\g,\ap & \gamma_\ap(1) & \gamma_\ap(2) & \gamma_\ap(3) & q_\ap(1) & q_\ap(2) & q_\ap(3) & h^* & r_\ap \\ \hline\hline
\GR{E}{6}, \ap_3 &  7/24 & 6/24 & 3/24 & 3 & 3 &1 & 12 & 1 \\
\GR{E}{7}, \ap_3 &  5/18 & 4/18 & 3/18 & 4 & 4 & 2 & 18 & 1 \\
\GR{E}{7}, \ap_5 &  11/36 & 10/36 & 3/36 & 5 & 5 & 1 & 18 & 1 \\
\GR{E}{8}, \ap_2 &  19/60 & 18/60 & 3/60 & 9 & 9 & 1 & 30 & 1 \\
\GR{E}{8}, \ap_8 &  17/60 & 14/60 & 9/60 & 7 & 7 & 3 & 30 & 1 \\
\GR{F}{4}, \ap_3 &  5/18 &  4/18 & 3/18 & 2 & 2 & 1 & 9 & 1 \\
\GR{G}{2}, \ap_1 &  5/24 & 2/24 & 9/24 & 1 & 1& 3 & 4 & 3 \\ \hline
\end{tabular}
\end{center}
\end{table}
\begin{table}[ht]
\caption{Exceptional Lie algebras, $d_\ap=4$}   \label{table:d=4}
\begin{center}
\begin{tabular}{>{$}c<{$}|>{$}c<{$}>{$}c<{$}>{$}c<{$}>{$}c<{$}|>{$}c<{$}>{$}c<{$}>{$}c<{$}>{$}c<{$}|>{$}c<{$}>{$}c<{$}}
\g,\ap  & \gamma_\ap(1) & \gamma_\ap(2) & \gamma_\ap(3) & \gamma_\ap(4) & q_\ap(1) & q_\ap(2) & q_\ap(3) 
 & q_\ap(4) & h^* & r_\ap \\ \hline\hline
\GR{E}{7}, \ap_4 &  4/18 & 4/18 & 3/18 & 2/18 & 2 &3 &2 & 1 & 18 & 1 \\
\GR{E}{8}, \ap_3 &  7/30 & 6/30 & 6/30 & 2/30 & 4 & 5 & 4 & 1 & 30 & 1 \\
\GR{E}{8}, \ap_6 &  13/60 & 14/60 & 9/60 & 8/60 & 3 & 5 & 3 & 2 & 30 & 1 \\
\GR{F}{4}, \ap_2 &  7/36 &  10/36 & 3/36 & 8/36 & 1 & 3 & 1 & 2 & 9 & 2 \\  \hline
\end{tabular}
\end{center}
\end{table}

\begin{table}[ht]
\caption{Exceptional Lie algebras, $d_\ap\ge 5$}   \label{table:d>5}
\begin{center}
\begin{tabular}{>{$}c<{$}|>{$}c<{$}>{$}c<{$}>{$}c<{$}>{$}c<{$}>{$}c<{$}>{$}c<{$}|>{$}c<{$}>{$}c<{$}}
\g,\ap  & \gamma_\ap(1) & \gamma_\ap(2) & \gamma_\ap(3) & \gamma_\ap(4) & \gamma_\ap(5) & \gamma_\ap(6) &  q_\ap(i) & d_\ap \\ \hline\hline
\GR{E}{8}, \ap_4 &  11/60 & 10/60 & 9/60 & 8/60 & 5/60 & - & 2,3,3,2,1 & 5  \\
\GR{E}{8}, \ap_5 &  9/60 & 10/60 & 9/60 & 8/60 & 5/60 & 6/60 & 1,2,2,2,1,1 & 6 \\ \hline
\end{tabular}
\end{center}
\end{table}

\vskip1ex
{\small {\bf Acknowledgements.}
Part of this work was done while it was still possible for me to enjoy a friendly environment 
of the Max-Planck-Institut f\"ur Mathematik (Bonn).

\end{document}